\documentclass[11pt]{article}

\usepackage[a4paper,margin=1in]{geometry}
\usepackage{amsmath,amssymb,amsthm,mathtools}
\usepackage{enumitem}
\usepackage{booktabs}
\usepackage{array}
\usepackage{microtype}
\usepackage{cite}
\usepackage{comment}
\usepackage[hidelinks]{hyperref}
\usepackage{authblk}
\usepackage{xcolor}

\hypersetup{
  pdftitle={Boundary–Moment Universality and Curvature Corrections in Random Geometric Graphs on Riemannian Manifolds},
  pdfauthor={Taro Hayashi, Subaru Nomoto, Ryoichi Suzuki},
  pdfsubject={Universal second-order three-vertex expansions, curvature-sensitive estimation, and the degree-two threshold law for random geometric graphs on closed Riemannian manifolds},
  pdfkeywords={random geometric graphs, Riemannian manifolds, scalar curvature, Euler characteristic, geometric estimation, U-statistics, degree-two maximum-degree threshold, Poisson approximation}
}

\title{Boundary–Moment Universality and Curvature Corrections in Random Geometric Graphs on Riemannian Manifolds}
\author[1]{Taro Hayashi\thanks{Corresponding author. E-mail: \texttt{haya4taro@gmail.com}. ORCID: 0009-0005-0145-8042}}
\author[1]{Subaru Nomoto\thanks{E-mail: \texttt{snomoto@fc.ritsumei.ac.jp}. ORCID: 0009-0009-0701-8370}}
\author[2]{Ryoichi Suzuki\thanks{E-mail: \texttt{rsuzukimath@gmail.com}. ORCID: 0000-0001-9979-1882}}

\affil[1]{Department of Mathematical Sciences, College of Science and Engineering, Ritsumeikan University, 1-1-1 Noji-higashi, Kusatsu, Shiga 525-8577, Japan.}
\affil[2]{Department of Business Economics, School of Management, Tokyo University of Science, 1-11-2 Fujimi, Chiyoda, Tokyo 102-0071, Japan.}
\date{}

\newtheorem{theorem}{Theorem}[section]
\newtheorem{proposition}[theorem]{Proposition}
\newtheorem{lemma}[theorem]{Lemma}
\newtheorem{corollary}[theorem]{Corollary}
\theoremstyle{remark}
\newtheorem{remark}[theorem]{Remark}

\newcommand{\R}{\mathbb R}
\newcommand{\Vol}{\operatorname{Vol}}
\newcommand{\Scal}{\operatorname{Scal}}
\newcommand{\Ric}{\operatorname{Ric}}
\newcommand{\inj}{\operatorname{inj}}
\newcommand{\grad}{\operatorname{grad}}
\newcommand{\ind}{\mathbf 1}
\newcommand{\E}{\mathbb E}
\newcommand{\Pp}{\mathbb P}
\newcommand{\eps}{\varepsilon}
\numberwithin{equation}{section}

\begin{document}
\maketitle

\begin{abstract}

Let $(M,g)$ be a smooth, closed, connected $d$-dimensional Riemannian manifold, and let $X_1,\ldots,X_n$ be i.i.d.\ with common law $f\,d\mathrm{vol}_g$, where $f\in C^4(M)$ is strictly positive.
We derive a uniform intrinsic second-order expansion for symmetric three-vertex edge-indicator statistics supported on connected configurations, including the induced-path and triangle kernels. 
The second-order term separates density variation, normal-coordinate Jacobians, and the curvature-induced motion of the internal-chord boundary. 
Within this three-vertex connected symmetric class, a universal
boundary--moment identity reduces the kernel-dependent contribution to a
common intrinsic functional involving
$\int_M f\|\grad f\|_g^2\,d\mathrm{vol}_g$ and
$\int_M f^3\operatorname{Scal}_g\,d\mathrm{vol}_g$.
For the normalized path--triangle contrast, the Euclidean leading term cancels. We construct a consistent estimator of this intrinsic functional and, 
in a denser bandwidth regime, we prove an exact-expectation-centered root-\(n\) central limit theorem via the first Hoeffding projection.
For uniform sampling on a closed surface, the estimator consistently recovers the Euler characteristic.
We also study the threshold radius at which the maximum degree of a binomial random geometric graph first reaches two.
Using the active-triple intensity expansion and a dependency-graph Poisson approximation, we obtain the order-$n^{-3/d}$ correction to the log-survival law for $d>6$.

\end{abstract}



\section{Introduction}

Random geometric graphs encode local metric information through pairwise
distance constraints \cite{Penrose2003}.  At scales $r\downarrow0$, a
neighborhood of a point on a smooth Riemannian manifold is approximately
Euclidean, so the leading behavior of many local graph statistics is
curvature-insensitive.
We ask how intrinsic geometry first enters second-order expansions of
symmetric three-vertex statistics supported on connected configurations and
how those corrections propagate to the degree-two maximum-degree threshold.

Write $[n]:=\{1,\ldots,n\}$.  Let $G_n(r)$ be the graph on the sample points
$\{X_1,\ldots,X_n\}$ in which two vertices are adjacent when their geodesic
distance is at most $r$, and let $\Delta(G)$ denote the maximum degree of a
graph $G$.  Define
\[
S_{2,n}:=\inf\{r>0:\Delta(G_n(r))\ge2\}.
\]
The natural sparse scale is
\[
r_{2,n}(t):=t n^{-3/(2d)},
\]
because the expected number of three-point degree-two witnesses is then of
order one; equivalently, $n^3r_{2,n}(t)^{2d}=t^{2d}$.  For $d>6$,
Theorem~\ref{thm:S2} gives the order-$n^{-3/d}$ intrinsic correction to the
survival law.  The main new geometric input is the curvature-induced motion
of the discontinuous internal-edge constraint.

\subsection{Relation to existing threshold and manifold expansions}
\label{subsec:relation-previous-work}

Classical Poisson limits for short-distance $U$-statistics and minimum
interpoint distances go back to
\cite{SilvermanBrown1978,KanagawaMochizukiTanaka1992}; general geometric
order-statistic and Poisson-approximation frameworks include
\cite{SchulteThale2012,DecreusefondSchulteThale2016}.
A recent preprint by Otsuka \cite{Otsuka2026} proves the leading fixed-$k$
threshold law in Euclidean space and a compound-Poisson limit for the point
process of degree-$k$ vertices.
For $k=2$, the feasible graphs are $P_3$
and $K_3$, and Otsuka's total Euclidean intensity is
\[
\mu^{\mathrm O}_{d,2}
=
A_2(d)\int_{\mathbb R^d}f^3\,dx.
\]
Thus his leading survival law agrees exactly with the tangent-space leading
term obtained below.  The two Poisson descriptions concern different
observables: Otsuka counts degree-two vertices, so a triangle contributes
three vertices, whereas our active-triple count assigns unit weight to each
$P_3$ or $K_3$ cluster.  Our contribution is the next-order curvature
correction on a closed Riemannian manifold.

First-order tangent-space limit theory for stabilizing point-process
functionals on manifolds was developed by Penrose and Yukich
\cite{PenroseYukich2013}; related manifold models include
\cite{BobrowskiOliveira2019,AuffingerLerarioLundberg2021,LerarioMulas2021}.
Subgraph counts and their Poisson and normal limits in Euclidean random
geometric graphs are classical \cite[Chapter~3]{Penrose2003}.
For our shrinking-radius kernels, we use the exact Hoeffding decomposition
\cite[Chapter~5]{Serfling1980} and verify projection dominance directly from
geometric overlap bounds.

There is also a distinct literature on recovering continuum curvature from
finite metric or graph data.  Hickok and Blumberg \cite{HickokBlumberg2023}
estimate scalar curvature from metric-ball volumes, while
\cite{vanDerHoornEtAl2023,GarciaTrillosWeber2026} study consistency of
Ollivier--Ricci curvature on random geometric graphs or data clouds.
Our target is instead the integrated density--curvature functional
$\mathcal G_3(f;g)$ obtained from two global motif counts; under uniform
sampling on a closed surface, Gauss--Bonnet converts it into the Euler
characteristic.  Scalar-curvature corrections for small geodesic balls are
classical \cite{Gray1973}, and related corrections occur in Riemannian
Poisson--Voronoi geometry \cite{CalkaChapronEnriquez2021}.
The new technical issue here is the curvature-induced motion of a
discontinuous internal-distance boundary.

\subsection{Contributions and scope}
\label{subsec:contributions-scope}

The main geometric result, Theorem~\ref{thm:three-point-master}, is a
second-order expectation expansion for any symmetric function of the three
edge indicators that vanishes on graphs with at most one edge.  This class is
precisely the two-dimensional span of the induced-path and triangle kernels.
The proof yields a stronger uniform anchored expansion, which is also used in
the threshold analysis.

The central structural result is
Proposition~\ref{prop:universal-boundary-moment}: the Jacobian and
moving-boundary coefficients are linked, and every admissible weight
$\mathfrak h$ in this three-vertex class satisfies
\[
\mathcal B_{\mathfrak h}(d)=\frac{d-1}{2}M_{\mathfrak h}(d).
\]
We prove this boundary--moment identity by a distributional divergence
argument in the displacement variable.

Here and throughout, $\grad h$ denotes the Riemannian gradient of a smooth
function $h$.  Define
\[
\mathcal G_3(f;g):=
\int_M f\|\grad f\|_g^2\,d\mathrm{vol}_g
+\frac16\int_M f^3\Scal_g\,d\mathrm{vol}_g.
\]
Theorem~\ref{thm:three-point-master} gives, for every admissible kernel,
\[
\mathbb E U_{\mathfrak h}(r)
=\binom n3 r^{2d}\left[
J_{\mathfrak h}(d)\int_M f^3\,d\mathrm{vol}_g
-\frac{3M_{\mathfrak h}(d)}{2d}r^2\mathcal G_3(f;g)
+o(r^2)\right].
\]
The scalar-curvature coefficient combines the normal-coordinate Jacobian
contribution $-M_{\mathfrak h}/(3d)$ with the moving-boundary contribution
$+M_{\mathfrak h}/(12d)$.  The latter is a one-sided shape derivative of a
discontinuous graph-induced kernel and is invisible in the Euclidean tangent
law.

For the active-triple kernel,
$J_{\mathfrak h_2}(d)/6=A_2(d)$ and
$M_{\mathfrak h_2}(d)/(4d)=C_2(d)$.
The normalized induced-path minus triangle contrast cancels the order-$r^{2d}$
Euclidean term and yields an estimator of $\mathcal G_3(f;g)$.
Theorem~\ref{thm:contrast-consistency} gives variance bounds and consistency,
while Theorem~\ref{thm:contrast-dense-clt} gives, in a denser bandwidth
regime, a root-$n$ central limit theorem centered at the exact expectation.
Without a quantitative higher-order bias expansion, we do not claim a
target-centered central limit theorem.  Under uniform sampling on a closed
surface, Corollary~\ref{cor:euler-inference} gives consistent recovery of the
Euler characteristic.

For the degree-two threshold, Theorem~\ref{thm:S2} gives the
order-$n^{-3/d}$ intrinsic correction to the log-survival law for $d>6$.
Its proof represents the survival event as a zero-count event for active
triples and uses a dependency-graph Chen--Stein bound to transfer the
expectation expansion to the zero-count law.  The restriction $d>6$ ensures
that the probabilistic approximation error is smaller than the geometric
$r^2$ correction.  We do not treat manifolds with boundary or the critical
dimension $d=6$, and we do not claim a general theorem for higher-degree
thresholds.

To the best of our knowledge, the one-sided moving-boundary coefficient,
the boundary--moment identity, and the zero-mass path--triangle estimator of
$\mathcal G_3(f;g)$ have not previously been derived.

Section~\ref{sec:setting} introduces the setting and
Section~\ref{sec:main-results} states the main results.
Sections~\ref{sec:preliminaries}--\ref{sec:statistical-inference}
develop the geometric, threshold, and statistical arguments.
The appendices contain auxiliary computations and trace details.

\section{Setting and notation}
\label{sec:setting}

Throughout, $(M,g)$ is a smooth, closed, connected $d$-dimensional Riemannian
manifold.  Let $d_g$ be the geodesic distance and $d\mathrm{vol}_g$ the
Riemannian volume measure, and write
\[
\Vol_g(M):=\int_M d\mathrm{vol}_g.
\]
Let
\[
X_1,\ldots,X_n\overset{\mathrm{i.i.d.}}{\sim} f\,d\mathrm{vol}_g,
\qquad
f\in C^4(M),\quad f>0,\quad \int_M f\,d\mathrm{vol}_g=1.
\]
For a smooth function $\varphi$ on $M$,
our Laplacian convention is
\[
\Delta_g\varphi=\operatorname{div}_g (\grad\varphi).
\]
Since $M$ is closed, integration by parts gives
\[
\int_M f\Delta_g f\,d\mathrm{vol}_g=-\int_M \|\grad f\|_g^2\,d\mathrm{vol}_g.
\]
Let $\nabla$ be the Levi-Civita connection with respect to $g$.
Our curvature convention is
\[
R(u,v)w
:=\nabla_u\nabla_vw-\nabla_v\nabla_uw-\nabla_{[u,v]}w,
\qquad
R_x(u,v,w,z):=\langle R_x(u,v)z,w\rangle_g.
\]
Thus $R_x(u,v,u,v)=\langle R_x(u,v)v,u\rangle_g$, and the unit
sphere has sectional curvature $+1$.

Let $\omega_d$ be the Euclidean volume of the unit ball in $\R^d$, and let
$d\sigma$ denote the standard unnormalized surface measure on $S^{d-1}$, so that
$\sigma(S^{d-1})=d\omega_d$.
Let
\[
\iota:=\inj(M,g)>0.
\]
Compactness also gives a uniform strong-convexity radius $\rho_{\mathrm c}>0$:
every ball $B_g(x,\rho_{\mathrm c})$ is strongly geodesically convex.  Put
\[
r_*:=\min\{\iota/10,\rho_{\mathrm c}/4\}.
\]
For a compact set $K\subset\mathbb R^d\times\mathbb R^d$, define
\[
R_K:=\sup_{(z,w)\in K}\max\{|z|,|z-w|\},\qquad
r_K:=\min\left\{r_*,\frac{\iota}{2(1+R_K)},
\frac{\rho_{\mathrm c}}{2(1+R_K)}\right\}.
\]
For a compactly supported cutoff $\chi$, we write $R_\chi$ and $r_\chi$
for the same quantities with $K=\operatorname{supp}\chi$.  Every local
statement below uses the corresponding support-dependent radius.  For
$x\in M$, let
\[
\mathcal F_x:=\{\tau:\mathbb R^d\to T_xM:\tau\text{ is an orthonormal
isometry}\},
\qquad
\mathcal F(M):=\{(x,\tau):x\in M,\ \tau\in\mathcal F_x\}.
\]
The orthonormal frame bundle $\mathcal F(M)$ is compact.  We use the notation
\[
\exp_{x,\tau}(z):=\exp_x(\tau z),
\qquad
R_{x,\tau}(u,v,w,z):=R_x(\tau u,\tau v,\tau w,\tau z).
\]
All estimates on fixed Euclidean compact sets are uniform in
$(x,\tau)\in\mathcal F(M)$.  The kernels used below are simultaneously
$O(d)$-invariant, so the resulting integrals are independent of the choice
of frame.

Local tangent variables are always defined by
\[
z_i=r^{-1}\exp_x^{-1}(x_i)
\]
inside the injectivity ball of $x$; thus $\exp_x^{-1}$ is used only locally.

We write $O(\cdot)$ and $o(\cdot)$ uniformly in the base point $x\in M$;
when a parameter $t$ appears in the threshold scale, all asymptotics are uniform
for $t$ in compact subsets of $[0,\infty)$.  Constants denoted by $C$ may change
from line to line and depend on $M,g,f,d$ and on the fixed compact range of $t$,
but never on $n$ or $r$.

The hypotheses $f\in C^4(M)$ and $f>0$ are standing assumptions adopted to
keep a single set of conditions throughout the paper; we do not optimize the
regularity or positivity assumptions theorem by theorem.  The $C^4$ condition
justifies the uniform normal-coordinate expansions used below.  Compactness
and positivity imply that the sampling density is bounded above and below away
from zero, so no tail or support-boundary localization issue arises.
All zero-count probability expansions are understood on compact $t$-ranges,
including $t=0$.  At $t=0$ the relevant cluster count is zero almost surely,
because the sampling law is nonatomic.  The displayed formulas are interpreted
using this endpoint convention and uniformity in $t$.

\section{Main results}
\label{sec:main-results}

The proofs of the statements in this section are deferred to the
subsequent sections.  After each main statement, we indicate where it is
established.

\subsection{Path and triangle graph functionals}
\label{subsec:three-point-statement}

Fix $\lambda_P,\lambda_K\in\mathbb R$ and define
\[
\mathfrak h(a,b,c)
:=
\lambda_P(ab+ac+bc)
+
(\lambda_K-3\lambda_P)abc,
\qquad
(a,b,c)\in\{0,1\}^3.
\]
Thus $\mathfrak h$ assigns weight $\lambda_P$ to a three-vertex path
$P_3$, weight $\lambda_K$ to a triangle $K_3$, and weight $0$ to every
disconnected three-vertex graph.  In particular, $\mathfrak h$ is symmetric
in its three arguments and
\[
\mathfrak h(a,b,c)=0
\qquad\text{whenever }a+b+c\le1.
\]
We refer to such a weight as an admissible symmetric weight.

Let $\binom{[n]}{3}$ be the family of all three-element subsets
$\{i,j,k\}\subset[n]$.  For
$\alpha=\{i,j,k\}\in\binom{[n]}{3}$, with $i<j<k$, put
\[
\xi_{\alpha,\mathfrak h}(r)
=\mathfrak h\!\left(
\ind_{\{d_g(X_i,X_j)\le r\}},
\ind_{\{d_g(X_i,X_k)\le r\}},
\ind_{\{d_g(X_j,X_k)\le r\}}
\right),
\qquad
U_{\mathfrak h}(r)
=\sum_{\alpha\in\binom{[n]}{3}}\xi_{\alpha,\mathfrak h}(r).
\]
For $z_1,z_2\in\mathbb R^d$, define
\[
H_{\mathfrak h}(z_1,z_2)
=\mathfrak h\!\left(
\ind_{\{|z_1|\le1\}},
\ind_{\{|z_2|\le1\}},
\ind_{\{|z_1-z_2|\le1\}}
\right),
\]
\[
\ell_{\mathfrak h}(a,b)
=\mathfrak h(a,b,1)-\mathfrak h(a,b,0),
\]
\[
\widetilde L_{\mathfrak h}(z_1,z_2)
:=\ell_{\mathfrak h}\!\left(
\ind_{\{|z_1|\le1\}},
\ind_{\{|z_2|\le1\}}
\right),
\qquad
L_{\mathfrak h}(z,w)
:=\widetilde L_{\mathfrak h}(z,z-w).
\]
Thus $\widetilde L_{\mathfrak h}$ takes two endpoint variables, whereas
$L_{\mathfrak h}$ takes an endpoint and a displacement.  The
connected-support assumption makes the following moments finite:
\[
J_{\mathfrak h}(d)
=\int_{(\mathbb R^d)^2}
H_{\mathfrak h}(z_1,z_2)\,dz_1\,dz_2,
\qquad
M_{\mathfrak h}(d)
=\int_{(\mathbb R^d)^2}
H_{\mathfrak h}(z_1,z_2)|z_1|^2\,dz_1\,dz_2.
\]
By symmetry in $z_1$ and $z_2$, equivalently,
\[
M_{\mathfrak h}(d)
=
\frac12
\int_{(\mathbb R^d)^2}
H_{\mathfrak h}(z_1,z_2)
\bigl(|z_1|^2+|z_2|^2\bigr)\,dz_1\,dz_2.
\]
For $d\ge2$,
\[
\mathcal B_{\mathfrak h}(d)
=\int_{S^{d-1}}\int_{\mathbb R^d}
L_{\mathfrak h}(z,w)
\bigl(|z|^2-(z\cdot w)^2\bigr)\,dz\,d\sigma(w).
\]

\begin{proposition}
\label{prop:universal-boundary-moment}
For every admissible symmetric weight $\mathfrak h$ and every $d\ge2$,
\[
\mathcal B_{\mathfrak h}(d)
=
\frac{d-1}{2}M_{\mathfrak h}(d).
\]
\end{proposition}

The proof is given in Section~\ref{sec:preliminaries}, following
Lemma~\ref{lem:BV-jump-formula}.  There the identity is obtained from the
$BV$ jump formula by a distributional divergence argument in the
displacement variable.

Define
\[
\mathcal G_3(f;g):=
\int_M f\|\grad f\|_g^2\,d\mathrm{vol}_g
+\frac16\int_M f^3\Scal_g\,d\mathrm{vol}_g.
\]
The next theorem records the common intrinsic functional for all admissible
weights, including signed linear combinations.

\begin{theorem}
\label{thm:three-point-master}
Assume $d\ge2$ and the standing hypotheses on $(M,g,f)$.
For every admissible symmetric weight $\mathfrak h$ and every $n\ge3$,
as $r\downarrow0$,
\[
\E U_{\mathfrak h}(r)
=
\binom n3 r^{2d}
\left[
J_{\mathfrak h}(d)\int_M f^3\,d\mathrm{vol}_g
-\frac{3M_{\mathfrak h}(d)}{2d}r^2\mathcal G_3(f;g)
+o(r^2)
\right].
\]
The remainder in the bracket is independent of $n$, so the expansion
also holds along arbitrary integer sequences $n=n(r)\ge3$.
\end{theorem}

The proof is given in Subsection~\ref{subsec:three-point-master}.
There we first establish the stronger uniform anchored expansion in
Proposition~\ref{prop:three-point-anchored}; integrating that expansion
over the anchor point yields Theorem~\ref{thm:three-point-master}.

\begin{remark}
Proposition~\ref{prop:universal-boundary-moment} reduces all second-order
kernel dependence to the two Euclidean moments
$J_{\mathfrak h}(d)$ and $M_{\mathfrak h}(d)$.
\end{remark}

\subsection{Statistical estimation from the path--triangle contrast}
\label{subsec:statistical-statements}

For the induced path and triangle, set
\[
\mathfrak h_{P_3}(a,b,c):=\ind_{\{a+b+c=2\}},
\qquad
\mathfrak h_{K_3}(a,b,c):=\ind_{\{a+b+c=3\}}.
\]
For $G\in\{P_3,K_3\}$, write
\[
H_G(z_1,z_2):=H_{\mathfrak h_G}(z_1,z_2),
\qquad
J_G(d):=J_{\mathfrak h_G}(d),
\qquad
M_G(d):=M_{\mathfrak h_G}(d).
\]
Define the normalized contrast kernel and statistic by
\[
\mathfrak c_d
:=\frac{\mathfrak h_{P_3}}{J_{P_3}(d)}
 -\frac{\mathfrak h_{K_3}}{J_{K_3}(d)},
\qquad
C_n(r):=U_{\mathfrak c_d}(r).
\]
Thus
\[
C_n(r)
=\frac{U_{P_3}(r)}{J_{P_3}(d)}
 -\frac{U_{K_3}(r)}{J_{K_3}(d)}.
\]
Define
\[
D_d
:=\frac{M_{P_3}(d)}{J_{P_3}(d)}
 -\frac{M_{K_3}(d)}{J_{K_3}(d)},
\qquad
\gamma_d:=\frac{3D_d}{2d}.
\]
The explicit Euclidean moment formulas proved later imply
$D_d>0$ and $\gamma_d>0$
for every $d\ge2$; see Lemma~\ref{lem:contrast-nondegeneracy}.
Since $J_{\mathfrak c_d}(d)=0$ and
$M_{\mathfrak c_d}(d)=D_d$, Theorem~\ref{thm:three-point-master}
gives
\begin{equation}
\E C_n(r)
=-\gamma_d\binom n3 r^{2d+2}\mathcal G_3(f;g)
+o\!\left(n^3r^{2d+2}\right).
\label{eq:contrast-mean}
\end{equation}
Define
\begin{equation}
\widehat{\mathcal G}_{3,n}(r)
:=-\frac{C_n(r)}
{\gamma_d\binom n3 r^{2d+2}}.
\label{eq:G3-estimator}
\end{equation}

\begin{theorem}
\label{thm:contrast-consistency}
There are $C<\infty$ and $r_0>0$ such that, for
$n\ge3$ and $0<r<r_0$,
\begin{equation}
\operatorname{Var} C_n(r)
\le C\left(
n^3r^{2d}
+n^4r^{3d}
+n^5r^{4d+4}
\right).
\label{eq:contrast-variance-bound}
\end{equation}
Consequently,
\begin{equation}
\operatorname{Var}\widehat{\mathcal G}_{3,n}(r)
\le C\left(
\frac1{n^3r^{2d+4}}
+\frac1{n^2r^{d+4}}
+\frac1n
\right).
\label{eq:G3-estimator-variance-bound}
\end{equation}
Let $r_n\downarrow0$ and assume
\begin{equation}
n^3r_n^{2d+4}\rightarrow\infty,
\qquad
n^2r_n^{d+4}\rightarrow\infty.
\label{eq:contrast-consistency-bandwidth}
\end{equation}
Then
\[
\widehat{\mathcal G}_{3,n}(r_n)
\xrightarrow{\Pp}\mathcal G_3(f;g).
\]
\end{theorem}

For clarity in the proof section, an equivalent reformulation of
Theorem~\ref{thm:contrast-consistency} is restated as
Theorem~\ref{thm:contrast-consistency-revise} in
Section~\ref{sec:statistical-inference}.  That theorem is proved there,
and hence Theorem~\ref{thm:contrast-consistency} follows from it.

To state the projection-dominant limit, put
\begin{equation}
q_f(x):=
\frac{1}{2d}\|\grad f(x)\|_g^2
+\frac1d f(x)\Delta_gf(x)
-\frac{1}{4d}f(x)^2\Scal_g(x),
\label{eq:qf-definition}
\end{equation}
and
\begin{equation}
\psi_f(x)
:=q_f(x)+\frac{3}{2d}\mathcal G_3(f;g).
\label{eq:psi-f-definition}
\end{equation}
Integration by parts gives
\[
\int_M q_f f\,d\mathrm{vol}_g
=-\frac{3}{2d}\mathcal G_3(f;g),
\qquad
\int_M\psi_f f\,d\mathrm{vol}_g=0.
\]

\begin{theorem}
\label{thm:contrast-dense-clt}
Let $r_n\downarrow0$ and suppose
\begin{equation}
n r_n^{d+4}\rightarrow\infty.
\label{eq:dense-bandwidth-condition}
\end{equation}
Then
\begin{equation}
\sqrt n\left(
\widehat{\mathcal G}_{3,n}(r_n)
-\E\widehat{\mathcal G}_{3,n}(r_n)
\right)
\xrightarrow{\mathcal D}
N(0,\sigma_f^2),
\label{eq:G3-dense-clt}
\end{equation}
where
\begin{equation}
\sigma_f^2
:=4d^2\int_M\psi_f(x)^2f(x)\,d\mathrm{vol}_g(x).
\label{eq:G3-asymptotic-variance}
\end{equation}
The limit is understood as degenerate when $\sigma_f^2=0$.
\end{theorem}

For clarity in the proof section, an equivalent reformulation of
Theorem~\ref{thm:contrast-dense-clt} is restated as
Theorem~\ref{thm:contrast-dense-clt-revise} in
Section~\ref{sec:statistical-inference}.  That theorem is proved there,
and hence Theorem~\ref{thm:contrast-dense-clt} follows from it.

\begin{remark}
The central limit theorem is centered at the exact expectation.
Since the second-order expansion provides no quantitative rate for
\[
\E\widehat{\mathcal G}_{3,n}(r_n)-\mathcal G_3(f;g),
\]
we do not claim a target-centered central limit theorem or confidence
intervals for $\mathcal G_3(f;g)$.
\end{remark}

\begin{corollary}
\label{cor:euler-inference}
Assume $d=2$ and $f\equiv V^{-1}$, where
$V:=\Vol_g(M)$.  Then
\[
J_{P_3}(2)=M_{P_3}(2)=\frac{9\sqrt3\pi}{4},
\quad
J_{K_3}(2)=\frac{\pi(4\pi-3\sqrt3)}4,
\quad
\gamma_2=\frac{\pi}{4\pi-3\sqrt3}.
\]
Let
\[
\widehat F_{3,n}(r)
:=\frac{U_{K_3}(r)}
{\binom n3 r^4J_{K_3}(2)},
\qquad
\widehat V_n(r):=
\begin{cases}
\widehat F_{3,n}(r)^{-1/2},
&\widehat F_{3,n}(r)>0,\\
1,
&\widehat F_{3,n}(r)=0.
\end{cases}
\]
Define
\begin{equation}
\widehat\chi_n(r)
:=\frac{3}{2\pi}\widehat V_n(r)^3
\widehat{\mathcal G}_{3,n}(r)
=-\frac{3(4\pi-3\sqrt3)}{2\pi^2}
\frac{\widehat V_n(r)^3C_n(r)}
{\binom n3r^6}.
\label{eq:euler-estimator}
\end{equation}
If $r_n\downarrow0$ satisfies
\begin{equation}
n^2r_n^6\rightarrow\infty,
\label{eq:surface-consistency-bandwidth}
\end{equation}
then
\[
\widehat V_n(r_n)\xrightarrow{\Pp}V,
\qquad
\widehat\chi_n(r_n)\xrightarrow{\Pp}\chi(M).
\]
Consequently, since $\chi(M)$ is integer-valued, rounding
$\widehat\chi_n(r_n)$ to the nearest integer recovers $\chi(M)$
with probability tending to one.
\end{corollary}

For clarity in the proof section, an equivalent reformulation of
Corollary~\ref{cor:euler-inference} is restated as
Corollary~\ref{cor:euler-inference-revise} in
Section~\ref{sec:statistical-inference}.  That corollary is proved there,
and hence Corollary~\ref{cor:euler-inference} follows from it.

\subsection{Degree-two threshold: \texorpdfstring{$k=2$}{k=2}}
\label{subsec:k2-statement}

For $z_1,z_2\in\R^d$, define
\[
a=\ind_{\{|z_1|\le1\}},
\qquad
b=\ind_{\{|z_2|\le1\}},
\qquad
c=\ind_{\{|z_1-z_2|\le1\}}.
\]
Define
\[
\mathfrak h_2(a,b,c):=\ind_{\{a+b+c\ge2\}}.
\]
The indicator that the three-point configuration
$\{0,z_1,z_2\}$ has maximum degree at least $2$ is
\[
H_2(z_1,z_2)
=\mathfrak h_2(a,b,c)
=ab+ac+bc-2abc.
\]
Define the Euclidean constants
\[
A_2(d)
:=\frac16\int_{(\mathbb R^d)^2}
H_2(z_1,z_2)\,dz_1\,dz_2,
\qquad
C_2(d)
:=\frac1{4d}\int_{(\mathbb R^d)^2}
H_2(z_1,z_2)|z_1|^2\,dz_1\,dz_2.
\]
Both are positive.  Explicit lens-volume and one-dimensional formulas are
given in Section~\ref{sec:k2} and
Appendix~\ref{app:k2-one-dimensional}.

\begin{theorem}
\label{thm:S2}
Under the standing assumptions, assume $d>6$ and define
$r_{2,n}(t):=t n^{-3/(2d)}$.
Then, as $n\to\infty$, for every $T<\infty$, uniformly for
$0\le t\le T$,
\[
\begin{aligned}
\log\Pp\{S_{2,n}>r_{2,n}(t)\}
={}&-A_2(d)t^{2d}\int_M f^3\,d\mathrm{vol}_g\\
&+C_2(d)t^{2d+2}n^{-3/d}
\left\{
\begin{aligned}
&\int_M f\|\grad f\|_g^2\,d\mathrm{vol}_g\\
&+\frac16\int_M f^3\Scal_g\,d\mathrm{vol}_g
\end{aligned}
\right\}
+o(n^{-3/d}).
\end{aligned}
\]
\end{theorem}

The proof is completed in Section~\ref{sec:k2}.
More precisely, Proposition~\ref{prop:N2-zero-count} gives the
corresponding zero-count expansion for the active-triple count
$N_2(r)$, and Theorem~\ref{thm:S2} follows from that proposition together
with the identity
\[
\{S_{2,n}>r\}=\{N_2(r)=0\}.
\]

\begin{corollary}
Under the hypotheses of Theorem~\ref{thm:S2}, if
$f\equiv \Vol_g(M)^{-1}$, then, for every $T<\infty$, uniformly for
$0\le t\le T$,
\[
\log\Pp\{S_{2,n}>t n^{-3/(2d)}\}
=-\frac{A_2(d)t^{2d}}{\Vol_g(M)^2}
+\frac{C_2(d)t^{2d+2}}{6\Vol_g(M)^3}
\left(
\int_M\Scal_g\,d\mathrm{vol}_g
\right)n^{-3/d}
+o(n^{-3/d}).
\]
\end{corollary}

This corollary follows immediately from Theorem~\ref{thm:S2} by taking
$f\equiv\Vol_g(M)^{-1}$.  Indeed,
\[
\grad f=0,
\qquad
\int_M f^3\,d\mathrm{vol}_g
=\Vol_g(M)^{-2},
\]
and
\[
\int_M f^3\Scal_g\,d\mathrm{vol}_g
=
\Vol_g(M)^{-3}
\int_M\Scal_g\,d\mathrm{vol}_g.
\]

\begin{remark}[Why $d>6$]
The geometric correction is of order $n^{-3/d}$, whereas the
largest dependency-graph error is $O(n^{-1/2})$.  Thus $d>6$ ensures
$n^{-1/2}=o(n^{-3/d})$.
\end{remark}

\section{Preliminaries}
\label{sec:preliminaries}

\subsection{Chord expansion and boundary variation}

Let $\mathcal F(M)$ denote the orthonormal frame bundle of $M$.  Thus an
element $(x,\tau)\in\mathcal F(M)$ consists of a point $x\in M$ and a
linear isometry
\[
\tau:\mathbb R^d\longrightarrow T_xM.
\]
We write
\[
\exp_{x,\tau}(u):=\exp_x(\tau u),
\]
and pull back the curvature tensor to $\mathbb R^d$ by setting
\[
R_{x,\tau}(a,b,c,d)
:=
R_x(\tau a,\tau b,\tau c,\tau d).
\]
Since $M$ is closed, the frame bundle $\mathcal F(M)$ is compact.

Define the rescaled two-point squared-distance function by
\[
\Phi_{r,x,\tau}(z,w)
:=
r^{-2}d_g\!\left(
\exp_{x,\tau}(rz),
\exp_{x,\tau}(r(z-w))
\right)^2.
\]
In Euclidean space, this function is exactly $|w|^2$.  We first identify
the exact fourth-order curvature correction and only then use Taylor's
theorem to control the remainder.  This distinction is important because
the moving-boundary coefficient depends on the precise quartic curvature
term.

All derivatives with respect to the variables
$u,v,z,w\in\mathbb R^d$ are ordinary Euclidean Fréchet derivatives.
For a linear map $A:\mathbb R^m\to\mathbb R^n$, we denote its operator
norm by
\[
\|A\|
:=
\sup_{|\xi|=1}|A\xi|,
\]
where $|\cdot|$ denotes the standard Euclidean norm.

\begin{lemma}
\label{lem:chord-expansion}
Let $K\subset\mathbb R^d\times\mathbb R^d$ be compact.
There exist constants $r_K>0$ and $C_K>0$ and a family of functions
$\mathcal R_{r,x,\tau}$, continuously differentiable in the $w$ variable,
such that, for
$0<r<r_K$, $(x,\tau)\in\mathcal F(M)$, and $(z,w)\in K$,
one has
\[
\Phi_{r,x,\tau}(z,w)
=
|w|^2
-\frac{r^2}{3}R_{x,\tau}(z,w,z,w)
+r^3\mathcal R_{r,x,\tau}(z,w),
\]
and
\[
\sup_{\substack{
0<r<r_K,\ (x,\tau)\in\mathcal F(M)\\
(z,w)\in K
}}
\left(
|\mathcal R_{r,x,\tau}(z,w)|
+
\|\mathrm D_w\mathcal R_{r,x,\tau}(z,w)\|
\right)
\le C_K.
\]
\end{lemma}

\begin{proof}
For a continuous map $f:[0,1]\to\mathbb R^d$, we define
$\|f\|_{C^0([0,1])}
:=
\sup_{0\le t\le1}|f(t)|$.
For $f\in C^1([0,1];\mathbb R^d)$, we set
$\|f\|_{C^1([0,1])}
:=
\|f\|_{C^0([0,1])}
+
\|\dot f\|_{C^0([0,1])}$.

For each $(x,\tau)\in\mathcal F(M)$, we use the normal coordinates induced
by $\tau$ and set
\[
D_{x,\tau}(u,v)
:=
d_g\!\left(
\exp_{x,\tau}(u),
\exp_{x,\tau}(v)
\right)^2.
\]
Since $M$ is closed, its convexity radius is positive.  Hence there exists
$s_0>0$ such that, for every $(x,\tau)\in\mathcal F(M)$ and every
$u,v\in\mathbb R^d$ satisfying $|u|+|v|<s_0$,
the points $\exp_{x,\tau}(u)$ and $\exp_{x,\tau}(v)$ lie in a common
strongly convex normal neighborhood.  In particular, the minimizing
geodesic between them is unique, and $D_{x,\tau}(u,v)$ depends smoothly
on $(x,\tau,u,v)$ in this region.
Fix $0<s_1<s_0$ and restrict to the compact parameter set
\[
\left\{
(x,\tau,u,v):
(x,\tau)\in\mathcal F(M),\quad
|u|+|v|\le s_1
\right\}.
\]
Since this parameter set is compact, all derivatives of
$D_{x,\tau}(u,v)$ needed below are uniformly bounded on it.

We first identify the fourth-order jet of $D_{x,\tau}$.  The standard
normal-coordinate
metric expansion (see, for example, \cite{Gray2004,Chavel2006,Petersen2016})
gives, under our curvature convention,
\begin{equation}
 g_y(a,b)=\langle a,b\rangle
 -\frac13R_{x,\tau}(a,y,b,y)
 +O(|y|^3|a||b|),
 \label{eq:normal-metric-two-point}
\end{equation}
uniformly in $(x,\tau)$. 
Let $q=v-u$, let $\lambda(t)=u+tq$, and let $\gamma$ be
the coordinate representation of the minimizing geodesic from $u$ to $v$, affinely parametrized on $[0,1]$. 
The normal-coordinate metric is uniformly equivalent to the Euclidean
metric on the restricted parameter set.  Since $\gamma$ is minimizing and
has constant speed,
\[
|\gamma'(t)|_{g_{\gamma(t)}}
=
d_g\!\left(\exp_{x,\tau}(u),\exp_{x,\tau}(v)\right)
\le |u|+|v|=:s.
\]
The uniform equivalence of the two metrics therefore gives
$\|\gamma'\|_{C^0([0,1])}=O(s)$.
Since $\gamma(t)
=
u+\int_0^t\gamma'(r)\,dr$,
we also have
$\|\gamma\|_{C^0([0,1])}=O(s)$.
Hence
\[
\|\gamma\|_{C^1([0,1])}=O(s).
\] 
In normal coordinates, the Christoffel symbols satisfy
$\Gamma(y)=O(|y|)$ uniformly in $(x,\tau)$
The geodesic equation gives
$|\gamma''(t)|
\le
C|\gamma(t)|\,|\gamma'(t)|^2
=
O(s^3)$,
and hence
\[
\|\gamma''\|_{C^0([0,1])}=O(s^3).
\]
We set $h:=\gamma-\lambda$.
Since $\lambda''=0$, we have $h''=\gamma''$ and $h(0)=h(1)=0$.
Using the endpoint conditions and integrating twice, one obtains
\[
\|h\|_{C^1([0,1])}
\le
C\|h''\|_{C^0([0,1])}.
\]
Consequently,
\[
\|h\|_{C^1([0,1])}
=
\|\gamma-\lambda\|_{C^1([0,1])}
=
O(s^3).
\]
Since $\gamma$ is a minimizing geodesic affinely parametrized on $[0,1]$,
it has constant speed, and its energy equals the square of its length.
Therefore,
\[
D_{x,\tau}(u,v)
=
\int_0^1
g_{\gamma(t)}\!\left(\gamma'(t),\gamma'(t)\right)\,dt.
\]
Applying \eqref{eq:normal-metric-two-point} with
$y=\gamma(t)$ and $a=b=\gamma'(t)$,
we obtain
\[
g_{\gamma(t)}\!\left(\gamma'(t),\gamma'(t)\right)
=
|\gamma'(t)|^2
-\frac13
R_{x,\tau}\!\left(
\gamma'(t),\gamma(t),\gamma'(t),\gamma(t)
\right)+O\!\left(
|\gamma(t)|^3|\gamma'(t)|^2
\right).
\]
Since $\|\gamma\|_{C^1([0,1])}=O(s)$,
the remainder term is uniformly $O(s^5)$.  Hence
\[
D_{x,\tau}(u,v)
=
\int_0^1|\gamma'(t)|^2\,dt-\frac13\int_0^1
R_{x,\tau}\!\left(
\gamma'(t),\gamma(t),\gamma'(t),\gamma(t)
\right)\,dt
+O(s^5).
\]
Since $\gamma=\lambda+h$ and $\gamma'=q+h'$,
\[
\begin{aligned}
\int_0^1|\gamma'(t)|^2\,dt
&=
\int_0^1|q+h'(t)|^2\,dt\\
&=
|q|^2
+
2\int_0^1\langle q,h'(t)\rangle\,dt
+
\int_0^1|h'(t)|^2\,dt.
\end{aligned}
\]
The middle term vanishes because
$\int_0^1\langle q,h'(t)\rangle\,dt
=
\langle q,h(1)-h(0)\rangle
=
0$.
Moreover, since $\|h\|_{C^1([0,1])}=O(s^3)$, we have
$\int_0^1|h'(t)|^2\,dt=O(s^6)$.
Therefore
\[
\int_0^1|\gamma'(t)|^2\,dt
=
|q|^2+O(s^6).
\]
By multilinearity,
\[
\begin{aligned}
R_{x,\tau}(\gamma',\gamma,\gamma',\gamma)
-R_{x,\tau}(q,\lambda,q,\lambda)
&=
R_{x,\tau}(\gamma'-q,\gamma,\gamma',\gamma)
+
R_{x,\tau}(q,\gamma-\lambda,\gamma',\gamma)\\
&\quad\qquad+
R_{x,\tau}(q,\lambda,\gamma'-q,\gamma)
+
R_{x,\tau}(q,\lambda,q,\gamma-\lambda).
\end{aligned}
\]
Since the curvature tensor is uniformly bounded,
\[
|q|=O(s),\qquad
\|\lambda\|_{C^0([0,1])}=O(s),\qquad
\|\gamma\|_{C^1([0,1])}=O(s),
\]
and
\[
\|\gamma-\lambda\|_{C^1([0,1])}
=
\|h\|_{C^1([0,1])}
=
O(s^3),
\]
each term on the right-hand side is $O(s^6)$, uniformly in
$t\in[0,1]$.  Hence
\[
R_{x,\tau}(\gamma',\gamma,\gamma',\gamma)
=
R_{x,\tau}(q,\lambda,q,\lambda)+O(s^6).
\]
Substituting these two estimates into the preceding expansion of
$D_{x,\tau}(u,v)$ gives
\[
D_{x,\tau}(u,v)
=
|q|^2
-\frac13\int_0^1
R_{x,\tau}(q,\lambda(t),q,\lambda(t))\,dt
+O(s^5).
\]
Since $\lambda(t)=u+tq$, the skew symmetries of the curvature tensor imply
\[
R_{x,\tau}(q,\lambda(t),q,\lambda(t))
=
R_{x,\tau}(q,u,q,u).
\]
Moreover, since $q=v-u$,
\[
R_{x,\tau}(q,u,q,u)
=
R_{x,\tau}(u,v,u,v).
\]
Therefore,
\[
D_{x,\tau}(u,v)
=
|u-v|^2
-\frac13R_{x,\tau}(u,v,u,v)
+O(s^5).
\]

Set
\[
E_{x,\tau}(u,v)
:=
D_{x,\tau}(u,v)-|u-v|^2
+\frac13R_{x,\tau}(u,v,u,v).
\]
Since $E_{x,\tau}(u,v)=O(s^5)$, all derivatives of
$E_{x,\tau}$ at $(0,0)$ of total order at most four vanish.
The fifth derivatives of $E_{x,\tau}$ are uniformly bounded on the
compact restricted parameter set.  Taylor's theorem with integral
remainder, applied also after one differentiation, therefore yields
uniformly on this compact subdomain
\[
|E_{x,\tau}(u,v)|\le Cs^5,
\qquad
|\mathrm D_{(u,v)}E_{x,\tau}(u,v)|\le Cs^4.
\]
For $(z,w)$ in a fixed compact set $K$, we set
$u=rz$ and $v=r(z-w)$.
Then
\[
s=r|z|+r|z-w|=O_K(r),
\qquad
|u-v|^2=r^2|w|^2.
\]
By the symmetries of the curvature tensor,
\[
\begin{aligned}
R_{x,\tau}(u,v,u,v)
&=
r^4R_{x,\tau}(z,z-w,z,z-w)\\
&=
r^4R_{x,\tau}(z,w,z,w).
\end{aligned}
\]
Hence
\[
\Phi_{r,x,\tau}(z,w)
=
|w|^2
-\frac{r^2}{3}R_{x,\tau}(z,w,z,w)
+r^{-2}E_{x,\tau}(rz,r(z-w)).
\]
Define
\[
\mathcal R_{r,x,\tau}(z,w)
:=
r^{-5}E_{x,\tau}(rz,r(z-w)).
\]
Then
$\Phi_{r,x,\tau}(z,w)
=
|w|^2
-\frac{r^2}{3}R_{x,\tau}(z,w,z,w)
+r^3\mathcal R_{r,x,\tau}(z,w)$,
and
$|\mathcal R_{r,x,\tau}(z,w)|\le C_K$.
By the chain rule,
\[
\mathrm D_w\mathcal R_{r,x,\tau}(z,w)
=
-r^{-4}
\mathrm D_vE_{x,\tau}(rz,r(z-w)).
\]
Since
$|\mathrm D_{(u,v)}E_{x,\tau}(u,v)|\le Cs^4$,
we have
\[
\left\|
\mathrm D_vE_{x,\tau}(rz,r(z-w))
\right\|
\le C_Kr^4.
\]
Therefore,
$\left\|
\mathrm D_w\mathcal R_{r,x,\tau}(z,w)
\right\|
\le C_K$.
This proves the lemma.
\end{proof}

\begin{corollary}
\label{cor:radial-chord-boundary}
Let $L\subset\mathbb R^d\times S^{d-1}$ be compact and let $\rho$ range over a fixed compact neighborhood
$[1-\varepsilon,1+\varepsilon]$ of $1$. Then, uniformly in
$(x,\tau)\in\mathcal F(M)$, $(z,\theta)\in L$, and $\rho$,
\[
\partial_\rho\Phi_{r,x,\tau}(z,\rho\theta)
=
2\rho
-\frac{2r^2}{3}\rho
R_{x,\tau}(z,\theta,z,\theta)
+O(r^3).
\]
Consequently, for all sufficiently small $r$, the equation
$\Phi_{r,x,\tau}(z,\rho\theta)=1$
has a unique solution $\rho_{r,x,\tau}(z,\theta)$ near $\rho=1$, and
\[
\rho_{r,x,\tau}(z,\theta)
=
1+\frac{r^2}{6}
R_{x,\tau}(z,\theta,z,\theta)
+O(r^3),
\]
uniformly in $(x,\tau)\in\mathcal F(M)$ and $(z,\theta)\in L$.
\end{corollary}

\begin{proof}
Apply Lemma~\ref{lem:chord-expansion} to the compact set
\[
K_{L,\varepsilon}
:=
\left\{
(z,\rho\theta):
(z,\theta)\in L,\ 
\rho\in[1-\varepsilon,1+\varepsilon]
\right\}.
\]
Since $R_{x,\tau}(z,\rho\theta,z,\rho\theta)
=
\rho^2R_{x,\tau}(z,\theta,z,\theta)$,
we have
\[
\Phi_{r,x,\tau}(z,\rho\theta)
=
\rho^2
-\frac{r^2}{3}\rho^2
R_{x,\tau}(z,\theta,z,\theta)
+r^3\mathcal R_{r,x,\tau}(z,\rho\theta).
\]
Differentiating with respect to $\rho$ gives
\[
\partial_\rho\Phi_{r,x,\tau}(z,\rho\theta)
=
2\rho
-\frac{2r^2}{3}\rho
R_{x,\tau}(z,\theta,z,\theta)
+r^3
\mathrm D_w\mathcal R_{r,x,\tau}(z,\rho\theta)[\theta].
\]
Since $|\theta|=1$ and $\mathrm D_w\mathcal R_{r,x,\tau}$ is uniformly
bounded, the last term is $O(r^3)$.  Thus
\[
\partial_\rho\Phi_{r,x,\tau}(z,\rho\theta)
=
2\rho
-\frac{2r^2}{3}\rho
R_{x,\tau}(z,\theta,z,\theta)
+O(r^3).
\]
Finally, set
\[
F_{r,x,\tau,z,\theta}(\rho)
:=
\Phi_{r,x,\tau}(z,\rho\theta)-1.
\]
On $[1-\varepsilon,1+\varepsilon]$, we have
$\partial_\rho F_{r,x,\tau,z,\theta}(\rho)=2\rho+O(r^2)$.
Thus,
after decreasing $r$ if necessary,
\[
\partial_\rho F_{r,x,\tau,z,\theta}(\rho)\ge1,
\]
uniformly in the remaining parameters.  Hence
$F_{r,x,\tau,z,\theta}$ is strictly increasing.
Set
\[
A_{x,\tau}(z,\theta)
:=
R_{x,\tau}(z,\theta,z,\theta).
\]
Since $\mathcal F(M)\times L$ is compact, there exists $A_L>0$ such that
\[
|A_{x,\tau}(z,\theta)|\le A_L.
\]
Choose $B>0$ sufficiently large.
Since $F_{r,x,\tau,z,\theta}(\rho)
=
\rho^2-1
-\frac{r^2}{3}\rho^2A_{x,\tau}(z,\theta)
+O(r^3)$,
we obtain, for sufficiently small $r$,
\[
F_{r,x,\tau,z,\theta}(1-Br^2)<0,
\qquad
F_{r,x,\tau,z,\theta}(1+Br^2)>0.
\]
The intermediate value theorem therefore gives a root
\[
\rho_{r,x,\tau}(z,\theta)
\in(1-Br^2,1+Br^2),
\]
and strict monotonicity shows that this root is unique.  Hence
\[
\left|
\rho_{r,x,\tau}(z,\theta)-1
\right|
\le Br^2
\]
uniformly in $(x,\tau)\in\mathcal F(M)$ and $(z,\theta)\in L$.
In particular,
$\rho_{r,x,\tau}(z,\theta)
=
1+O(r^2)$,
where the $O(r^2)$ term is uniform in these parameters.
We write
\[
\rho_{r,x,\tau}(z,\theta)=1+\delta.
\]
Then $\delta=O(r^2)$ uniformly in the same parameters, and substitution into
$\Phi_{r,x,\tau}(z,\rho\theta)=1$
gives
\[
2\delta
-\frac{r^2}{3}R_{x,\tau}(z,\theta,z,\theta)
+O\!\left(\delta^2+r^2|\delta|+r^3\right)
=0.
\]
Since
$\delta^2+r^2|\delta|=O(r^4)$,
we conclude that
\[
2\delta
-\frac{r^2}{3}R_{x,\tau}(z,\theta,z,\theta)
+O(r^3)
=0.
\]
Therefore,
\[
\delta
=
\frac{r^2}{6}R_{x,\tau}(z,\theta,z,\theta)
+O(r^3),
\]
and hence
$\rho_{r,x,\tau}(z,\theta)
=
1+\frac{r^2}{6}
R_{x,\tau}(z,\theta,z,\theta)
+O(r^3)$.
\end{proof}

For the remainder of this subsection, for each $x$ choose any
$\tau\in\mathcal F_x$ and suppress $\tau$ from the notation.  The resulting
statements are independent of this choice and remain uniform over
$(x,\tau)\in\mathcal F(M)$.

The Ricci and scalar traces in the next lemma use the convention fixed in
Section~\ref{sec:setting}:
\[
\Ric_R(w,w)=\sum_i R(e_i,w,e_i,w),\qquad
\Scal(R)=\sum_i\Ric_R(e_i,e_i).
\]

\begin{lemma}\label{lem:rotational-contraction}
Let $\psi:\R^d\times S^{d-1}\to\R$ be bounded, measurable, compactly supported
in the first variable, and invariant under simultaneous rotations:
\[
\psi(Oz,Ow)=\psi(z,w),\qquad O\in O(d).
\]
Then, for $d\ge2$ and every algebraic curvature tensor $R$ on $\R^d$,
\[
\int_{S^{d-1}}\int_{\R^d}\psi(z,w)R(z,w,z,w)\,dz\,d\sigma(w)
=\frac{\Scal(R)}{d(d-1)}
\int_{S^{d-1}}\int_{\R^d}\psi(z,w)
\bigl(|z|^2-(z\cdot w)^2\bigr)\,dz\,d\sigma(w).
\]
\end{lemma}

\begin{proof}
Fix $w\in S^{d-1}$ and write
\[
z=sw+u,\qquad s=z\cdot w,\quad u\in w^\perp.
\]
By the alternating symmetries of $R$,
\[
R(z,w,z,w)=R(u,w,u,w),
\qquad |z|^2-(z\cdot w)^2=|u|^2.
\]
For fixed $w$, simultaneous rotational invariance implies that the restriction
of $\psi(sw+u,w)$ to $w^\perp$ is radial in $u$.  
Using the Euclidean identification
$\mathbb R^d\simeq(\mathbb R^d)^*$, we regard $u\otimes u$ as the
rank-one endomorphism
\[
v\longmapsto \langle u,v\rangle u.
\]
The resulting tensor integral is invariant under every orthogonal
transformation fixing $w$, and hence it is a scalar multiple of
$\Pi_{w^\perp}$.  Taking its trace determines the scalar and gives
\[
\int_{\R}\int_{w^\perp}\psi(sw+u,w)u\otimes u\,du\,ds
=\frac{A(w)}{d-1}\,\Pi_{w^\perp},
\]
where $\Pi_{w^\perp}$ is the orthogonal projection onto $w^\perp$ and
\[
A(w):=\int_{\R}\int_{w^\perp}\psi(sw+u,w)|u|^2\,du\,ds.
\]
Simultaneous rotational invariance implies that $A(w)$ is independent of
$w$. Denote its common value by $A$.
Choose an orthonormal basis
\[
e_1=w,e_2,\ldots,e_d,
\qquad e_2,\ldots,e_d\in w^\perp.
\]
Contracting the preceding tensor identity with
\[
(v_1,v_2)\longmapsto R(v_1,w,v_2,w),
\]
we obtain
\[
\begin{aligned}
\int_{\R^d}\psi(z,w)R(z,w,z,w)\,dz
&=
\frac{A}{d-1}
\sum_{\alpha=2}^d R(e_\alpha,w,e_\alpha,w)\\
&=
\frac{A}{d-1}\Ric_R(w,w).
\end{aligned}
\]
The last equality follows from $e_1=w$ and
$R(w,w,w,w)=0$.
Integrating over $S^{d-1}$ and using
\[
\int_{S^{d-1}}\Ric_R(w,w)\,d\sigma(w)
=\frac{\sigma(S^{d-1})}{d}\Scal(R)
=\omega_d\Scal(R)
\]
gives
\[
\int_{S^{d-1}}\int_{\R^d}\psi(z,w)R(z,w,z,w)\,dz\,d\sigma(w)
=\frac{A\omega_d}{d-1}\Scal(R).
\]
On the other hand,
\[
\int_{S^{d-1}}\int_{\R^d}\psi(z,w)(|z|^2-(z\cdot w)^2)\,dz\,d\sigma(w)
=d\omega_d A.
\]
Combining the last two displays yields the stated coefficient
$1/(d(d-1))$.
\end{proof}

Let $d_{\mathrm{prod}}$ denote the distance induced by the product metric on $\mathbb R^d\times S^{d-1}$. For $p\in\mathbb R^d\times S^{d-1}$ and $A\subset\mathbb R^d\times S^{d-1}$, we write 
\[ \operatorname{dist}_{\mathrm{prod}}(p,A) := \inf_{q\in A}d_{\mathrm{prod}}(p,q). \]
Define
\[ T_1 := \{(z,\theta)\in\mathbb R^d\times S^{d-1}:|z|=1\}, \qquad T_2 := \{(z,\theta)\in\mathbb R^d\times S^{d-1}:|z-\theta|=1\}. \]
In the definition of $T_2$, both $z$ and $\theta$ are regarded
as vectors in the same Euclidean space $\mathbb R^d$.
The defining functions
$F_1(z,\theta)=|z|^2-1$ and
$F_2(z,\theta)=|z-\theta|^2-1$ have nonzero differentials on their zero
sets.  
At an intersection point, $|z|=|z-\theta|=1$ implies $z\cdot\theta=1/2$. 
The $S^{d-1}$-tangent component of $dF_2$ is $-2(z-(z\cdot\theta)\theta)$, which is nonzero there, whereas the corresponding component of $dF_1$ vanishes. 
Thus the hypersurfaces meet transversely. 

For a compact set
$K\subset\mathbb R^d\times S^{d-1}$ and a measurable set
$A\subset\mathbb R^d\times S^{d-1}$, we write
\[
\operatorname{meas}_K(A)
:=
\int_{A\cap K} dz\,d\sigma(\theta).
\]

\begin{lemma}
\label{lem:fixed-trace-tubular}
Assume $d\ge2$, and let
$K\subset\mathbb R^d\times S^{d-1}$ be compact.
Then there exist constants $C_K<\infty$ and $\eta_K>0$ such that
\[
\operatorname{meas}_K
\left\{
p\in\mathbb R^d\times S^{d-1}:
\operatorname{dist}_{\mathrm{prod}}(p,T_1\cup T_2)<\eta
\right\}
\le C_K\eta,
\qquad
0<\eta<\eta_K.
\]
\end{lemma}

\begin{proof}
Since $T_1$ and $T_2$ are compact smooth embedded hypersurfaces,
the tubular neighborhood theorem implies that, for each $i=1,2$,
there exist constants $C_{K,i}<\infty$ and $\eta_i>0$ such that
\[
\operatorname{meas}_K
\left\{
p:
\operatorname{dist}_{\mathrm{prod}}(p,T_i)<\eta
\right\}
\le C_{K,i}\eta,
\qquad
0<\eta<\eta_i.
\]
Since
$\operatorname{dist}_{\mathrm{prod}}(p,T_1\cup T_2)
=
\min\left\{
\operatorname{dist}_{\mathrm{prod}}(p,T_1),
\operatorname{dist}_{\mathrm{prod}}(p,T_2)
\right\}$,
\[
\left\{
p:
\operatorname{dist}_{\mathrm{prod}}(p,T_1\cup T_2)<\eta
\right\}
=
\bigcup_{i=1}^2
\left\{
p:
\operatorname{dist}_{\mathrm{prod}}(p,T_i)<\eta
\right\}.
\]

Set
\[
\eta_K:=\min\{\eta_1,\eta_2\}.
\]
Then, for $0<\eta<\eta_K$, the subadditivity of
$\operatorname{meas}_K$ gives
\[
\begin{aligned}
&\operatorname{meas}_K
\left\{
p:
\operatorname{dist}_{\mathrm{prod}}(p,T_1\cup T_2)<\eta
\right\}\\
&\qquad\le
\sum_{i=1}^2
\operatorname{meas}_K
\left\{
p:
\operatorname{dist}_{\mathrm{prod}}(p,T_i)<\eta
\right\}\\
&\qquad\le
(C_{K,1}+C_{K,2})\eta.
\end{aligned}
\]
Thus the result follows with
\[
C_K:=C_{K,1}+C_{K,2}.
\]
\end{proof}

For $(x,\tau)\in\mathcal F(M)$, $r>0$, and
$z,w\in\mathbb R^d$, we put 
\[
c_{r,x,\tau}^{\mathrm{loc}}(z,w)
:=
\ind_{\{\Phi_{r,x,\tau}(z,w)\le1\}},
\qquad
c(w):=\ind_{\{|w|\le1\}}.
\]

\begin{lemma}
\label{lem:moving-shell}
Let $\chi\in C_c^\infty(\R^d\times\R^d)$ be fixed.
There exist constants
$0<r_0\le r_\chi$ and $0<K_0<\infty$, depending only on
$\chi$ and $(M,g)$, such that, for every
$(x,\tau)\in\mathcal F(M)$ and $0<r<r_0$,
\[
c_{r,x,\tau}^{\mathrm{loc}}(z,w)\ne c(w)
\quad\Longrightarrow\quad
\bigl||w|-1\bigr|\le K_0r^2
\]
for all $(z,w)\in\operatorname{supp}\chi$.
\end{lemma}

\begin{proof}
The two indicators agree when $w=0$.  For $w\ne0$, write
$w=\rho\theta$ with $\theta\in S^{d-1}$.  By
Lemma~\ref{lem:chord-expansion}, uniformly in $(x,\tau)$ and on
$\operatorname{supp}\chi$,
\[
\Phi_{r,x,\tau}(z,\rho\theta)
=\rho^2-\frac{r^2}{3}\rho^2R_{x,\tau}(z,\theta,z,\theta)+O(r^3).
\]
On the compact support of $\chi$, after decreasing $r_0$ if necessary, the
entire perturbation satisfies
$|\Phi_{r,x,\tau}(z,\rho\theta)-\rho^2|\le Cr^2$.

Suppose that $|\rho^2-1|>2Cr^2$.
If $\rho^2-1>2Cr^2$,
then
\[
\begin{aligned}
\Phi_{r,x,\tau}(z,\rho\theta)-1
&=
(\rho^2-1)
+
\bigl(
\Phi_{r,x,\tau}(z,\rho\theta)-\rho^2
\bigr)\\
&\ge
(\rho^2-1)
-
\left|
\Phi_{r,x,\tau}(z,\rho\theta)-\rho^2
\right|\\
&>
2Cr^2-Cr^2\\
&>0.
\end{aligned}
\]
Thus both $\Phi_{r,x,\tau}(z,\rho\theta)$ and $\rho^2$ are greater
than one.
Similarly, if $\rho^2-1<-2Cr^2$,
then
\[
\begin{aligned}
\Phi_{r,x,\tau}(z,\rho\theta)-1
&\le
(\rho^2-1)
+
\left|
\Phi_{r,x,\tau}(z,\rho\theta)-\rho^2
\right|\\
&<
-2Cr^2+Cr^2\\
&<0.
\end{aligned}
\]
Thus both $\Phi_{r,x,\tau}(z,\rho\theta)$ and $\rho^2$ are less than
one.
Consequently, whenever $|\rho^2-1|>2Cr^2$,
the quantities
\[
\Phi_{r,x,\tau}(z,\rho\theta)-1
\quad\text{and}\quad
\rho^2-1
\]
have the same sign, and hence the two indicators agree.  Therefore,
if the two indicators differ, then
\[
|\rho^2-1|\le2Cr^2.
\]
Since $\rho\ge0$, we have
$|\rho^2-1|
=
|\rho-1|(\rho+1)
\ge
|\rho-1|$.
It follows that
\[
\bigl||w|-1\bigr|
=
|\rho-1|
\le
2Cr^2.
\]
Thus the assertion holds with $K_0:=2C$.
\end{proof}

Let
\[
B:=\{z\in\mathbb R^d:|z|\le1\},
\qquad
a(z):=\ind_B(z),
\qquad
b(z,w):=\ind_B(z-w).
\]
For a function $\ell:\{0,1\}^2\to\mathbb R$, put
\[
L_\ell(z,w):=\ell(a(z),b(z,w)),
\qquad
\|\ell\|_\infty
:=\max_{(\alpha,\beta)\in\{0,1\}^2}
|\ell(\alpha,\beta)|.
\]
If $\ell(0,0)=0$, then
\[
L_\ell(z,w)\ne0
\quad\Longrightarrow\quad
z\in B\cup(w+B).
\]
In particular, for $\theta\in S^{d-1}$,
\[
L_\ell(z,\theta)\ne0
\quad\Longrightarrow\quad
|z|\le2,\qquad |z-\theta|\le2.
\]
For $\chi\in C_c^\infty(\mathbb R^d\times\mathbb R^d)$, we set
\[
I_{r,x,\tau}^{\ell,\chi}
:=
\int_{\mathbb R^d}\int_{\mathbb R^d}
\chi(z,w)L_\ell(z,w)
\bigl(c_{r,x,\tau}^{\mathrm{loc}}(z,w)-c(w)\bigr)
\,dw\,dz
\]
and
\[
J_{x,\tau}^{\ell}
:=
\int_{S^{d-1}}\int_{\mathbb R^d}
L_\ell(z,\theta)
R_{x,\tau}(z,\theta,z,\theta)
\,dz\,d\sigma(\theta).
\]

\begin{lemma}
\label{lem:graph-boundary}
Assume $d\ge2$.  Let
$\ell:\{0,1\}^2\to\mathbb R$ satisfy $\ell(0,0)=0$.
Let $\chi\in C_c^\infty(\mathbb R^d\times\mathbb R^d)$ satisfy
$0\le\chi\le1$ and suppose that $\chi=1$ on a neighborhood of
\[
\mathcal S_*:=
\{(z,\theta)\in\mathbb R^d\times S^{d-1}:
 |z|\le2,\ |z-\theta|\le2\}.
\]
Then there exists $r_\chi>0$ such that, for every $L<\infty$,
\[
\sup_{\substack{(x,\tau)\in\mathcal F(M)\\
\|\ell\|_\infty\le L,\ \ell(0,0)=0}}
\frac{1}{r^2}
\left|
I_{r,x,\tau}^{\ell,\chi}
-\frac{r^2}{6}J_{x,\tau}^{\ell}
\right|
\longrightarrow0
\qquad (r\downarrow0).
\]
Moreover,
\[
\int_{\mathbb R^d}\int_{\mathbb R^d}
\chi(z,w)|L_\ell(z,w)|
\left|
c_{r,x,\tau}^{\mathrm{loc}}(z,w)-c(w)
\right|
\,dw\,dz
\le C_\chi\|\ell\|_\infty r^2
\]
for all $0<r<r_\chi$, uniformly in
$(x,\tau)\in\mathcal F(M)$.
For sufficiently small $r$, both integrals are independent of the
choice of $\chi$ among cutoffs satisfying the above condition.
\end{lemma}

\begin{proof}
The condition $\ell(0,0)=0$ implies that, on $|w|=1$, the support of
$L_\ell$ lies in $\mathcal S_*$.  We regard
$\mathcal S_*$ as a compact subset of
$\mathbb R^d\times\mathbb R^d$ via
$S^{d-1}\subset\mathbb R^d$.  Since $\chi=1$ on an open neighborhood
of $\mathcal S_*$, there exists $\delta_\chi>0$ such that
$\chi=1$ on the $\delta_\chi$-neighborhood of $\mathcal S_*$ in
$\mathbb R^d\times\mathbb R^d$.

\emph{Step 1: shell localization and the cutoff margin.}
By Lemma~\ref{lem:moving-shell}, the difference
$c_{r,x,\tau}^{\mathrm{loc}}-c$ is supported in
\[
1-K_0r^2\le |w|\le1+K_0r^2.
\]
Write $w=\rho\theta$.  If $L_\ell(z,w)\ne0$, then either $|z|\le1$ or
$|z-w|\le1$.  In the first case $(z,\theta)\in\mathcal S_*$, and hence
\[
\operatorname{dist}_{\mathbb R^{2d}}
\bigl((z,\rho\theta),\mathcal S_*\bigr)
\le |\rho-1|.
\]
In the second case, we put
$z_*:=z+(1-\rho)\theta$.
Then
$|z_*-\theta|=|z-w|\le1$
and 
$|z_*|\le2$. Thus $(z_*,\theta)\in\mathcal S_*$.  Therefore
\[
\begin{aligned}
\operatorname{dist}_{\mathbb R^{2d}}
\bigl((z,\rho\theta),\mathcal S_*\bigr)
&\le
\left|(z,\rho\theta)-(z_*,\theta)\right|\\
&=
\sqrt2\,|\rho-1|.
\end{aligned}
\]
Thus, in either case,
\[
\operatorname{dist}_{\mathbb R^{2d}}
\bigl((z,\rho\theta),\mathcal S_*\bigr)
\le
\sqrt2\,K_0r^2.
\]
After decreasing $r_0$ so that
$\sqrt2\,K_0r_0^2<\delta_\chi$, the cutoff is one on every nonzero
shell contribution.

\emph{Step 2: fixed reference traces.}
Let $K$ be a fixed compact set containing all $(z,\theta)$ for which
$(z,\rho\theta)$ belongs to $\operatorname{supp}\chi$ and
$|\rho-1|\le2K_0r_0^2$.  At $\rho=1$, the reference discontinuity traces are
\[
T_1=\{(z,\theta):|z|=1\},\qquad
T_2=T_{2,1}=\{(z,\theta):|z-\theta|=1\}.
\]
For general $\rho$, the second discontinuity surface is
$T_{2,\rho}:=\{(z,\theta):|z-\rho\theta|=1\}$; thus $T_2$ is used only as a
fixed reference trace.  By Lemma~\ref{lem:fixed-trace-tubular},
\[
\operatorname{meas}_K\{\operatorname{dist}_{\mathrm{prod}}
\bigl((z,\theta),T_1\cup T_2\bigr)<\eta\}
\le C\eta,
\qquad 0<\eta<\min\{1,\eta_K\}.
\]
Define
\[
G_\eta
:=
\left\{
(z,\theta)\in K:
\operatorname{dist}_{\mathrm{prod}}
\bigl((z,\theta),T_1\cup T_2\bigr)\ge\eta
\right\}.
\]

\emph{Step 3: constancy on the good set.}

For $(z,\theta)\in G_\eta$, we have
\[
\operatorname{dist}_{\mathrm{prod}}
\bigl((z,\theta),T_2\bigr)\ge\eta.
\]
To obtain an upper bound for the same distance, observe that
\[
\operatorname{dist}_{\mathrm{prod}}
\bigl((z,\theta),T_2\bigr)
\le \bigl||z-\theta|-1\bigr|.
\]
Indeed, if $z\ne\theta$, set
\[
z'
:=
\theta+\frac{z-\theta}{|z-\theta|}.
\]
Then $(z',\theta)\in T_2$, and hence
\[
\begin{aligned}
\operatorname{dist}_{\mathrm{prod}}
\bigl((z,\theta),T_2\bigr)
&\le
\operatorname{dist}_{\mathrm{prod}}
\bigl((z,\theta),(z',\theta)\bigr)\\
&=
|z-z'|\\
&=
\bigl||z-\theta|-1\bigr|.
\end{aligned}
\]
If $z=\theta$, choose any $e\in S^{d-1}$ and put $z'=\theta+e$;
the same inequality follows.  Therefore,
\[
\bigl||z-\theta|-1\bigr|\ge\eta.
\]
For fixed $\eta>0$, after decreasing $r$ if necessary so that
$2K_0r^2\le\eta/2$, every
\[
1-2K_0r^2\le\rho\le1+2K_0r^2
\]
satisfies
\[
\begin{aligned}
\bigl||z-\rho\theta|-1\bigr|
&\ge
\bigl||z-\theta|-1\bigr|-|\rho-1|\\
&\ge
\eta-2K_0r^2
\ge \frac{\eta}{2}.
\end{aligned}
\]
Hence the indicator
\[
b(z,\rho\theta)=\ind_{\{|z-\rho\theta|\le1\}}
\]
does not change as $\rho$ varies over this interval.  Since
$a(z)=\ind_{\{|z|\le1\}}$ is independent of $\rho$, it follows that
\[
L_\ell(z,\rho\theta)=L_\ell(z,\theta)
\]
throughout the enlarged radial interval.

\emph{Step 4: the exact signed shell identity.}
Consider the enlarged radial interval
\[
1-2K_0r^2\le\rho\le1+2K_0r^2.
\]
At either endpoint, $|\rho-1|>K_0r^2$, so
Lemma~\ref{lem:moving-shell} forces agreement of the Euclidean and Riemannian
indicators.  At the lower endpoint both equal one, and at the upper endpoint
both equal zero.  Moreover,
\[
\partial_\rho\Phi_{r,x,\tau}(z,\rho\theta)=2\rho+O(r^2)\ge1
\]
throughout the interval for small $r$.  Thus there is exactly one crossing,
and Corollary~\ref{cor:radial-chord-boundary} gives
\[
\rho_{r,x,\tau}(z,\theta)
=1+\frac{r^2}{6}R_{x,\tau}(z,\theta,z,\theta)+O(r^3)
\]
uniformly.

Polar Fubini and the uniqueness of this crossing give, for every
$(z,\theta)\in G_\eta$, the exact signed identity
\begin{align}
&\int_0^\infty \chi(z,\rho\theta)L_\ell(z,\rho\theta)
\left(\ind_{\{\Phi_{r,x,\tau}(z,\rho\theta)\le1\}}
      -\ind_{\{\rho\le1\}}\right)\rho^{d-1}\,d\rho
\notag\\
&\qquad=
\int_1^{\rho_{r,x,\tau}(z,\theta)}
\chi(z,\rho\theta)L_\ell(z,\rho\theta)\rho^{d-1}\,d\rho.
\label{eq:exact-signed-shell}
\end{align}
The right-hand side is interpreted as a signed integral when
$\rho_{r,x,\tau}<1$.  By the cutoff margin from Step~1 and the constancy from
Step~3, it equals
\[
L_\ell(z,\theta)
\frac{\rho_{r,x,\tau}(z,\theta)^d-1}{d}
=\frac{r^2}{6}L_\ell(z,\theta)
R_{x,\tau}(z,\theta,z,\theta)+O(\|\ell\|_\infty r^3)
\]
uniformly on $G_\eta$.

\emph{Step 5: uniform control of the bad set and the order of limits.}
Set
\[
B_\eta:=K\setminus G_\eta.
\]
Let $\mathcal E_{r,x,\tau,\eta}$ denote the contribution from $B_\eta$
to the full shell integral, minus $r^2/6$ times the contribution from
$B_\eta$ to the limiting surface integral.
The bad-set shell contribution is bounded by
$C\|\ell\|_\infty\eta r^2$: the shell has thickness $O(r^2)$ and the bad set
has measure $O(\eta)$.  Bounded curvature gives the same bound for the omitted
part of the limiting surface integral.  Consequently, for every $L<\infty$,
\[
\limsup_{r\downarrow0}
\sup_{(x,\tau)\in\mathcal F(M)}
\sup_{\|\ell\|_\infty\le L}
 r^{-2}|\mathcal E_{r,x,\tau,\eta}|\le C L\eta.
\]
For fixed $\eta$, the good-set calculation based on
\eqref{eq:exact-signed-shell} has an integrated remainder
$O(Lr^3)$.  We first let $r\downarrow0$ and subsequently let
$\eta\downarrow0$.  This proves the asserted $o(r^2)$ formula uniformly over
bounded families of $\ell$.  
Step~1 also shows that, for sufficiently small $r$, the integral is independent of the admissible cutoff.
\end{proof}

\begin{remark}
The boundary calculation in Lemma~\ref{lem:graph-boundary} does not rely
on a general transfer principle for discontinuous weights.  Such a principle
would in general require specifying the appropriate one-sided trace at the
moving boundary.
In the present graph-induced setting, however,
\[
H^{\mathrm{loc}}_{r,x}-H_{\mathfrak h}
=
L_{\mathfrak h}(z_1,z_1-z_2)
\bigl(c^{\mathrm{loc}}_{r,x}-c\bigr),
\]
and the coefficient $L_{\mathfrak h}$ depends only on the two fixed
anchor-edge indicators, not on the moving indicator $c$.
Consequently, its discontinuities lie only on fixed trace hypersurfaces.
After removing an $\eta$-neighborhood of these hypersurfaces,
$L_{\mathfrak h}$ is locally constant across the moving shell, so the
usual shell calculation applies.  The omitted contribution is
$O(\eta)$ by Lemma~\ref{lem:fixed-trace-tubular}.
Letting first $r\downarrow0$ and then $\eta\downarrow0$ yields the desired
boundary formula without any ambiguity concerning one-sided traces.
\end{remark}

\subsection{The boundary--moment identity}

We extend $\mathfrak h$ to $[0,1]^3$ by the multilinear polynomial
\[
\overline{\mathfrak h}(s_1,s_2,s_3)
:=
\lambda_P(s_1s_2+s_1s_3+s_2s_3)
+
(\lambda_K-3\lambda_P)s_1s_2s_3.
\]
Thus $\overline{\mathfrak h}$ is the unique multilinear extension of
$\mathfrak h$, and
\[
\overline{\mathfrak h}(a,b,c)
=
\mathfrak h(a,b,c)
\qquad
\text{for }(a,b,c)\in\{0,1\}^3.
\]

We write $BV_c(\R^{2d})$ for the space of functions of bounded variation
on $\R^{2d}$ with compact support.  For $u\in BV_c(\R^{2d})$ we split the
gradient measure as $Du=(D_zu,D_wu)$, so that $D_wu$ is the $\R^d$-valued
Radon measure formed by the last $d$ components of $Du$.  Throughout,
$\mathcal H^{d-1}$ denotes the $(d-1)$-dimensional Hausdorff measure on
$\R^d$; in particular $\sigma=\mathcal H^{d-1}\lfloor S^{d-1}$.

\begin{lemma}
\label{lem:BV-jump-formula}
Assume $d\ge2$.  Use the displacement coordinates
$z=z_1$ and
$w=z_1-z_2$.
We set
$a=\ind_{\{|z|\le1\}}$,
$b=\ind_{\{|z-w|\le1\}}$,
$c=\ind_{\{|w|\le1\}}$,
and
$\widehat H_{\mathfrak h}(z,w)
:=
\overline{\mathfrak h}(a,b,c)$.
Then
\[
\widehat H_{\mathfrak h}\in BV_c(\R^{2d})
\]
holds. 
Moreover, for every
$\varphi\in C_c^1(\R^{2d};\R^d)$, we have
\begin{align}
\int_{\R^{2d}}
\varphi\cdot dD_w\widehat H_{\mathfrak h}
={}&
\int_{\R^d}\int_{\{|z-w|=1\}}
\Delta_b\mathfrak h(a,c)\,
\varphi(z,w)\cdot(z-w)
\,d\mathcal H^{d-1}(w)\,dz
\label{eq:BV-jump-measure}\\
&-
\int_{\R^d}\int_{S^{d-1}}
\Delta_c\mathfrak h(a,b)\,
\varphi(z,w)\cdot w
\,d\sigma(w)\,dz,
\notag
\end{align}
where
\[
\begin{aligned}
\Delta_b\mathfrak h(a,c)
&:=
\mathfrak h(a,1,c)-\mathfrak h(a,0,c)\\
&=
\lambda_P(a+c)+(\lambda_K-3\lambda_P)ac,
\end{aligned}
\]
and
\[
\begin{aligned}
\Delta_c\mathfrak h(a,b)
&:=
\mathfrak h(a,b,1)-\mathfrak h(a,b,0)\\
&=
\lambda_P(a+b)+(\lambda_K-3\lambda_P)ab.
\end{aligned}
\]
The two measures on the right-hand side of \eqref{eq:BV-jump-measure}
are mutually singular; consequently, each of them is supported in
$\operatorname{supp}\widehat H_{\mathfrak h}$.
\end{lemma}

\begin{proof}
By the definition of $\overline{\mathfrak h}$,
$\widehat H_{\mathfrak h}
=
\lambda_P(ab+ac+bc)
+
(\lambda_K-3\lambda_P)abc$.
The functions $ab,ac,bc,abc$ are indicators of bounded convex sets in
$\R^{2d}$, hence belong to $BV_c(\R^{2d})$.  Therefore
\[
\widehat H_{\mathfrak h}\in BV_c(\R^{2d}).
\]

Fix $z\ne0$.  As functions of $w$,
\[
b=\ind_{B(z,1)},\qquad c=\ind_{B(0,1)},
\]
and Gauss--Green gives
\[
D_wb
=
(z-w)\,
\mathcal H^{d-1}\lfloor\{|z-w|=1\},
\qquad
D_wc
=
-w\,
\mathcal H^{d-1}\lfloor S^{d-1}.
\]
Moreover,
\[
\mathcal H^{d-1}
\bigl(\partial B(z,1)\cap S^{d-1}\bigr)=0.
\]
Hence the $BV$ product rule yields
\[
D_w(bc)=c\,D_wb+b\,D_wc.
\]
Since $a$ is constant in $w$,
\[
\begin{aligned}
D_w\bigl[\widehat H_{\mathfrak h}(z,\cdot)\bigr]
&=
\lambda_Pa(D_wb+D_wc)
+
\{\lambda_P+(\lambda_K-3\lambda_P)a\}
(c\,D_wb+b\,D_wc) \\
&=
\Delta_b\mathfrak h(a,c)\,D_wb
+
\Delta_c\mathfrak h(a,b)\,D_wc.
\end{aligned}
\]
This holds for every $z\ne0$.

Now let $\varphi\in C_c^1(\R^{2d};\R^d)$.  By the definition of
$D_w$ and Fubini,
\[
\int_{\R^{2d}}\varphi\cdot dD_w\widehat H_{\mathfrak h}
=
\int_{\R^d}
\left(
\int_{\R^d}
\varphi(z,\cdot)\cdot
dD_w[\widehat H_{\mathfrak h}(z,\cdot)]
\right)dz.
\]
The exceptional slice $z=0$ is $dz$-null.  Substituting the preceding
formula and the expressions for $D_wb,D_wc$ gives
\eqref{eq:BV-jump-measure}.

Finally, the two measures are concentrated on
\[
\Sigma_b=\{|z-w|=1\},
\qquad
\Sigma_c=\{|w|=1\}.
\]
For every $z\ne0$ their intersections in the $w$-slice are
$\mathcal H^{d-1}$-null; hence
\[
\Sigma_b\cap\Sigma_c
\]
is null for both measures, so the two measures are mutually singular.
Since their sum is $D_w\widehat H_{\mathfrak h}$, which vanishes on
$\R^{2d}\setminus\operatorname{supp}\widehat H_{\mathfrak h}$,
mutual singularity implies that each measure vanishes there separately.
\end{proof}
Note that, by the symmetry of $\mathfrak h$,
$\Delta_b\mathfrak h=\Delta_c\mathfrak h=\ell_{\mathfrak h}$; the two
symbols merely indicate which internal chord is moving.

\begin{proof}[Proof of Proposition~\ref{prop:universal-boundary-moment}]
Use the displacement coordinates $z=z_1$ and $w=z_1-z_2$, and let $\widehat H_{\mathfrak h}$ be as in Lemma~\ref{lem:BV-jump-formula}. Since 
$\widehat H_{\mathfrak h}(z,w) = H_{\mathfrak h}(z,z-w)$, 
the measure-preserving change of variables 
$(z_1,z_2)=(z,z-w)$ gives 
\[ M_{\mathfrak h}(d) = \int_{\mathbb R^{2d}} \widehat H_{\mathfrak h}(z,w)|z|^2\,dz\,dw. \]
Define 
\[ V(z,w) := |z|^2w-(z\cdot w)z. \]
Then
\[
\operatorname{div}_w V
=
d|z|^2-|z|^2
=
(d-1)|z|^2.
\]
Choose $\chi\in C_c^1(\mathbb R^{2d})$ such that $\chi=1$ on a
neighborhood of $\operatorname{supp}\widehat H_{\mathfrak h}$.  The
distributional product rule gives
\[
\operatorname{div}_w
\bigl(V\widehat H_{\mathfrak h}\bigr)
=
\widehat H_{\mathfrak h}\operatorname{div}_w V
+
V\cdot D_w\widehat H_{\mathfrak h}.
\]
Testing this identity against $\chi$, and using that
$\nabla_w\chi=0$ on $\operatorname{supp}\widehat H_{\mathfrak h}$, we obtain
\[
0
=
\int_{\mathbb R^{2d}}
\widehat H_{\mathfrak h}(z,w)
\operatorname{div}_w V(z,w)\,dz\,dw
+
\int_{\mathbb R^{2d}}
V(z,w)\cdot dD_w\widehat H_{\mathfrak h}(z,w).
\]
Since 
$M_{\mathfrak h}(d)
=
\int_{\mathbb R^{2d}}
\widehat H_{\mathfrak h}(z,w)|z|^2\,dz\,dw$,
the first term is
\[
(d-1)M_{\mathfrak h}(d).
\]

We next evaluate the second term using
Lemma~\ref{lem:BV-jump-formula}.  Since the measure
$D_w\widehat H_{\mathfrak h}$ is supported where $\chi=1$, we may apply
the lemma with $\varphi=\chi V$ and obtain
\begin{align*}
\int_{\mathbb R^{2d}}
V\cdot dD_w\widehat H_{\mathfrak h}
={}&
\int_{\mathbb R^d}\int_{\{|z-w|=1\}}
\Delta_b\mathfrak h(a,c)\,
V(z,w)\cdot(z-w)
\,d\mathcal H^{d-1}(w)\,dz\\
&-
\int_{\mathbb R^d}\int_{S^{d-1}}
\Delta_c\mathfrak h(a,b)\,
V(z,w)\cdot w
\,d\sigma(w)\,dz.
\end{align*}
Define
\[
I_b
:=
\int_{\mathbb R^d}\int_{\{|z-w|=1\}}
\Delta_b\mathfrak h(a,c)\,
V(z,w)\cdot(z-w)
\,d\mathcal H^{d-1}(w)\,dz.
\]
On $|w|=1$,
\[
\Delta_c\mathfrak h(a,b)
=
\mathfrak h(a,b,1)-\mathfrak h(a,b,0)
=
L_{\mathfrak h}(z,w),
\]
and
\[
V(z,w)\cdot w
=
|z|^2|w|^2-(z\cdot w)^2
=
|z|^2-(z\cdot w)^2.
\]
Hence, by the definition of $\mathcal B_{\mathfrak h}(d)$,
\[
\int_{\mathbb R^{2d}}
V\cdot dD_w\widehat H_{\mathfrak h}
=
I_b-\mathcal B_{\mathfrak h}(d).
\]
Consequently,
\begin{equation}
0
=
(d-1)M_{\mathfrak h}(d)
+
I_b
-
\mathcal B_{\mathfrak h}(d).
\label{eq:BV-divergence-ledger}
\end{equation}

It remains to compute $I_b$.  Parameterize the hypersurface
$\{|z-w|=1\}$ by
\[
u=z-w\in S^{d-1},
\qquad
w=z-u.
\]
Then
\[
I_b
=
\int_{S^{d-1}}\int_{\mathbb R^d}
\Delta_b\mathfrak h(a,c)\,
V(z,z-u)\cdot u
\,dz\,d\sigma(u).
\]
Apply the measure-preserving change of variables
\[
(z,u)\mapsto(z',w')=(-z,-u).
\]
Under this change,
\[
a=\ind_{\{|z'|\le1\}},
\qquad
c
=
\ind_{\{|z-u|\le1\}}
=
\ind_{\{|z'-w'|\le1\}}.
\]
By symmetry of $\mathfrak h$,
\[
\begin{aligned}
\Delta_b\mathfrak h(a,c)
&=
\mathfrak h(a,1,c)-\mathfrak h(a,0,c)\\
&=
\mathfrak h(a,c,1)-\mathfrak h(a,c,0)\\
&=
L_{\mathfrak h}(z',w').
\end{aligned}
\]
Moreover, since $|u|=1$,
\[
\begin{aligned}
V(z,z-u)\cdot u
&=
|z|^2(z-u)\cdot u
-
\bigl(z\cdot(z-u)\bigr)(z\cdot u)\\
&=
|z|^2(z\cdot u-1)
-
\bigl(|z|^2-z\cdot u\bigr)(z\cdot u)\\
&=
-|z|^2+(z\cdot u)^2\\
&=
-\bigl(|z'|^2-(z'\cdot w')^2\bigr).
\end{aligned}
\]
Therefore
\[I_b
=
-\mathcal B_{\mathfrak h}(d).\]
Substituting this into \eqref{eq:BV-divergence-ledger} gives
$0
=
(d-1)M_{\mathfrak h}(d)
-
2\mathcal B_{\mathfrak h}(d)$,
and hence
$\mathcal B_{\mathfrak h}(d)
=
\frac{d-1}{2}M_{\mathfrak h}(d)$.
\end{proof}

\subsection{Dependency-graph Poisson approximation and zero-count transfer}

We use Theorem~1 of Arratia--Goldstein--Gordon
\cite{ArratiaGoldsteinGordon1989}; see also
\cite{BarbourHolstJanson1992}.
The following lemma records our total-variation convention and the
simplification \(b_3=0\) furnished by a dependency graph.

\begin{lemma}
\label{lem:dep-graph}
Let \(I\) be finite, and let \((\xi_i)_{i\in I}\) be Bernoulli random
variables with dependency graph \((I,\sim)\), meaning that whenever
\(A,B\subset I\) are disjoint and no edge joins a vertex of \(A\) to a
vertex of \(B\), the families
$(\xi_i)_{i\in A}$ and $(\xi_j)_{j\in B}$
are independent.
Put
\[
p_i=\E\xi_i,
\qquad
W=\sum_{i\in I}\xi_i,
\qquad
\lambda=\E W,
\]
and let $N(i):=\{i\}\cup\{j:j\sim i\}$.
Define
\[
b_1
:=
\sum_{i\in I}\sum_{j\in N(i)}p_ip_j,
\qquad
b_2
:=
\sum_{i\in I}
\sum_{\substack{j\in N(i)\\j\ne i}}
\E[\xi_i\xi_j].
\]
With the convention
\[
d_{\mathrm{TV}}(\mu,\nu)
:=
\sup_{A\subseteq\mathbb Z_{\ge0}}
|\mu(A)-\nu(A)|,
\]
one has
\[
d_{\mathrm{TV}}
\bigl(
\mathcal L(W),
\operatorname{Po}(\lambda)
\bigr)
\le
\frac{1-e^{-\lambda}}{\lambda}(b_1+b_2)
\le
(1\wedge\lambda^{-1})(b_1+b_2),
\]
where the factor is interpreted as \(1\) when \(\lambda=0\).
\end{lemma}

The next elementary consequence transfers a Poisson approximation to the
zero-count probability.  This is the form needed later for the survival
event of the degree-two threshold.

\begin{lemma}
\label{lem:zero-transfer}
Let \(W_n\) be nonnegative integer-valued random variables and put
$\lambda_n:=\E W_n$.
Suppose that \((\lambda_n)\) is uniformly bounded and that
\[
d_{\mathrm{TV}}
\bigl(
\mathcal L(W_n),
\operatorname{Po}(\lambda_n)
\bigr)
\le
C\eps_n.
\]
Then
\[
\Pp\{W_n=0\}
=
e^{-\lambda_n}
+
O(\eps_n).
\]

Moreover, suppose that
$a_n\downarrow0$,
$\eps_n=o(a_n)$,
and
$\lambda_n
=
\lambda_0+\lambda_1a_n+o(a_n)$.
Then
\[
\log\Pp\{W_n=0\}
=
-\lambda_0-\lambda_1a_n+o(a_n).
\]

The logarithmic conclusion is uniform on a parameter set \(K\) whenever
the total-variation bound, the upper bound on \(\lambda_n\), and the
expansion of \(\lambda_n\) are all uniform on \(K\).
\end{lemma}

\begin{proof}
By the definition of total variation,
\[
\left|
\Pp\{W_n=0\}
-
e^{-\lambda_n}
\right|
\le
C\eps_n.
\]
Since \((\lambda_n)\) is uniformly bounded, there is \(c>0\) such that
\[
e^{-\lambda_n}\ge c
\]
for all \(n\).  Hence
\[
\Pp\{W_n=0\}
=
e^{-\lambda_n}
\bigl(1+O(\eps_n)\bigr).
\]

For the second assertion,
$\eps_n=o(a_n)$ and $a_n\downarrow0$
imply \(\eps_n\to0\).  Therefore
\[
\begin{aligned}
\log\Pp\{W_n=0\}
&=
-\lambda_n
+
\log\bigl(1+O(\eps_n)\bigr)\\
&=
-\lambda_n+O(\eps_n)\\
&=
-\lambda_0-\lambda_1a_n+o(a_n).
\end{aligned}
\]
The same argument is uniform on \(K\) under the stated uniform
assumptions.
\end{proof}

\section{Path--triangle expansion and proof of the \texorpdfstring{$k=2$}{k=2} theorem}
\label{sec:k2}

\subsection{Euclidean moment identities}

Let
\[
L_d(s):=2\omega_{d-1}\int_{s/2}^{1}(1-u^2)^{(d-1)/2}\,du,
\qquad 0\le s\le2,
\]
the volume of the intersection of two Euclidean unit balls whose centers are
a distance $s$ apart, and define
\[
T_0(d):=d\omega_d\int_0^1s^{d-1}L_d(s)\,ds,
\qquad
T_2(d):=d\omega_d\int_0^1s^{d+1}L_d(s)\,ds.
\]
For related geometric results on intersections of geodesic balls, see
\cite{CsikosHorvath2011}.  In this notation,
\[
A_2(d)=\frac12\omega_d^2-\frac13T_0(d),
\qquad
C_2(d)=\frac1{4d}\left(\frac{4d}{d+2}\omega_d^2-2T_2(d)\right).
\]
The positive one-dimensional representation of $C_2(d)$ is proved in
Appendix~\ref{app:k2-one-dimensional}.

For later use, write
\[
M_2(d)
:=
\int_{(\R^d)^2}
H_2(z_1,z_2)|z_1|^2\,dz_1\,dz_2.
\]
Thus, by the definition of $C_2(d)$,
\[
M_2(d)=4d\,C_2(d).
\]
We now record the tensor identities used in the three-point expansion.
They hold for every admissible symmetric three-vertex weight.

\begin{lemma}
\label{lem:three-point-tensor-moments}
Let \(\mathfrak h\) be an admissible symmetric weight. Then
\[
\int_{(\R^d)^2}
H_{\mathfrak h}(z_1,z_2)\,
z_1\otimes z_1\,dz_1\,dz_2
=
\frac{M_{\mathfrak h}(d)}{d}I_d,
\]
and
\[
\int_{(\R^d)^2}
H_{\mathfrak h}(z_1,z_2)\,
z_1\otimes z_2\,dz_1\,dz_2
=
\frac{M_{\mathfrak h}(d)}{2d}I_d.
\]
Consequently, for every symmetric bilinear form \(Q\) on \(\R^d\),
\[
\int_{(\R^d)^2}
H_{\mathfrak h}(z_1,z_2)
\bigl\{Q(z_1,z_1)+Q(z_2,z_2)\bigr\}
\,dz_1\,dz_2
=
\frac{2M_{\mathfrak h}(d)}{d}\operatorname{tr}Q.
\]
\end{lemma}

\begin{proof}
Throughout the proof, all unqualified integrals are over
\((\R^d)^2\) with respect to \(dz_1\,dz_2\).
The kernel \(H_{\mathfrak h}\) is invariant under simultaneous
orthogonal transformations and under relabeling of the three points
\(0,z_1,z_2\).  In particular, the measure-preserving transformations
\[
(z_1,z_2)\longmapsto(z_2,z_1),
\qquad
(z_1,z_2)\longmapsto(-z_1,z_2-z_1)
\]
permute the three edge indicators.  Hence
\[
\int H_{\mathfrak h}|z_1|^2
=
\int H_{\mathfrak h}|z_2|^2
=
\int H_{\mathfrak h}|z_1-z_2|^2
=
M_{\mathfrak h}(d).
\]

Since
$|z_1-z_2|^2
=
|z_1|^2+|z_2|^2-2z_1\cdot z_2$,
we obtain
\[
\int H_{\mathfrak h}z_1\cdot z_2
=
\frac12M_{\mathfrak h}(d).
\]

Set
\[
T_{11}
:=
\int H_{\mathfrak h}z_1\otimes z_1,
\qquad
T_{12}
:=
\int H_{\mathfrak h}z_1\otimes z_2.
\]
For every \(R\in O(d)\), simultaneous orthogonal invariance gives
\[
H_{\mathfrak h}(Ru_1,Ru_2)
=
H_{\mathfrak h}(u_1,u_2).
\]
Moreover, since \(R\) is orthogonal,
\[
|\det R|=1,
\]
so the change of variables
\[
(z_1,z_2)=(Ru_1,Ru_2)
\]
preserves Lebesgue measure.  Hence
\[
\begin{aligned}
T_{11}
&=
\int
H_{\mathfrak h}(Ru_1,Ru_2)
(Ru_1)\otimes(Ru_1)\,du_1\,du_2\\
&=
\int
H_{\mathfrak h}(u_1,u_2)
(Ru_1)\otimes(Ru_1)\,du_1\,du_2\\
&=
R\left(
\int
H_{\mathfrak h}(u_1,u_2)
u_1\otimes u_1\,du_1\,du_2
\right)R^{\mathsf T}\\
&=
RT_{11}R^{\mathsf T}.
\end{aligned}
\]
Similarly,
\[
\begin{aligned}
T_{12}
&=
\int
H_{\mathfrak h}(Ru_1,Ru_2)
(Ru_1)\otimes(Ru_2)\,du_1\,du_2\\
&=
R\left(
\int
H_{\mathfrak h}(u_1,u_2)
u_1\otimes u_2\,du_1\,du_2
\right)R^{\mathsf T}\\
&=
RT_{12}R^{\mathsf T}.
\end{aligned}
\]
Hence both \(T_{11}\) and \(T_{12}\) commute with every orthogonal
transformation.  Therefore there exist \(\alpha,\beta\in\R\) such that
\[
T_{11}=\alpha I_d,
\qquad
T_{12}=\beta I_d.
\]

Taking traces and using
\[
\operatorname{tr}(u\otimes v)=u\cdot v,
\]
we obtain
\[
d\alpha
=
\int H_{\mathfrak h}|z_1|^2
=
M_{\mathfrak h}(d),\qquad
d\beta
=
\int H_{\mathfrak h}z_1\cdot z_2
=
\frac12M_{\mathfrak h}(d).
\]
Thus
\[
T_{11}
=
\frac{M_{\mathfrak h}(d)}{d}I_d,
\qquad
T_{12}
=
\frac{M_{\mathfrak h}(d)}{2d}I_d.
\]

Finally, contracting the first identity with \(Q\) gives
\[
\int H_{\mathfrak h}Q(z_1,z_1)
=
\frac{M_{\mathfrak h}(d)}{d}\operatorname{tr}Q.
\]
By the symmetry under
\((z_1,z_2)\mapsto(z_2,z_1)\),
the same formula holds with \(z_1\) replaced by \(z_2\).  Adding the
two identities yields
\[
\int H_{\mathfrak h}
\bigl\{Q(z_1,z_1)+Q(z_2,z_2)\bigr\}
=
\frac{2M_{\mathfrak h}(d)}{d}\operatorname{tr}Q.
\]
\end{proof}

\subsection{Proof and specializations of the path--triangle expansion}
\label{subsec:three-point-master}

We first prove a stronger anchored form of
Theorem~\ref{thm:three-point-master}.  Its integrated consequence, written
in normalized form below, recovers the expansion in that theorem and provides
the uniform remainder needed later.

For $x\in M$ and $0<r<r_*$, define the normalized anchored integral
\[
\mathcal I_{\mathfrak h}(r,x)
:=r^{-2d}f(x)\int_{M^2}
\mathfrak h\!\Big(
\ind_{\{d_g(x,y)\le r\}},
\ind_{\{d_g(x,z)\le r\}},\ind_{\{d_g(y,z)\le r\}}
\Big)
f(y)f(z)\,d\mathrm{vol}_g(y)\,d\mathrm{vol}_g(z).
\]

\begin{proposition}
\label{prop:three-point-anchored}
Assume \(d\ge2\) and the standing hypotheses on \((M,g,f)\).
There exists \(r_0>0\) such that, for every \(0\le H<\infty\), there is a
nondecreasing function
\[
\omega_H:[0,r_0]\longrightarrow[0,\infty),
\qquad
\omega_H(s)\longrightarrow0
\quad\text{as }s\downarrow0,
\]
for which
\[
\begin{aligned}
\mathcal I_{\mathfrak h}(r,x)
={}&
J_{\mathfrak h}(d)f(x)^3\\
&+
\frac{M_{\mathfrak h}(d)}{d}r^2
\left\{
\frac12 f(x)\|\grad f(x)\|_g^2
+
f(x)^2\Delta_g f(x)
-
\frac14 f(x)^3\Scal_g(x)
\right\}+
\operatorname{Rem}_{\mathfrak h}(r,x),
\end{aligned}
\]
where
\[
\bigl|
\operatorname{Rem}_{\mathfrak h}(r,x)
\bigr|
\le
r^2\omega_H(r)
\]
for every \(x\in M\), every \(0<r\le r_0\), and every admissible
symmetric weight \(\mathfrak h\) satisfying
\(\|\mathfrak h\|_\infty\le H\).

Consequently, for every \(n\ge3\), uniformly over such \(\mathfrak h\),
\[
\begin{aligned}
\frac{\E U_{\mathfrak h}(r)}
{\binom{n}{3}r^{2d}}
={}&
J_{\mathfrak h}(d)\int_M f^3\,d\mathrm{vol}_g\\
&-
\frac{3M_{\mathfrak h}(d)}{2d}r^2
\left\{
\int_M f\|\grad f\|_g^2\,d\mathrm{vol}_g
+
\frac16\int_M f^3\Scal_g\,d\mathrm{vol}_g
\right\}
+
o_H(r^2),
\end{aligned}
\]
where the remainder is independent of \(n\).
\end{proposition}

\begin{proof}
Fix the first labeled point as the anchor \(x\) and choose
\(\tau\in\mathcal F_x\).
A nonzero \(\mathfrak h\)-weight requires at least two of the three possible
edges, because \(\mathfrak h\) vanishes on configurations with at most one
edge.  Hence every contributing three-vertex graph is either a path or a
triangle and has graph diameter at most two.  Since every graph edge has
geodesic length at most \(r\), the triangle inequality shows that each of
the other two vertices lies in \(B_g(x,2r)\).

Choose \(r_0>0\) so small that
\[
2r_0<\min\{\iota,\rho_{\mathrm c}\}.
\]
Then, for \(0<r\le r_0\), the other two points of every contributing
configuration have unique representations
$\exp_{x,\tau}(rz_1)$ and $\exp_{x,\tau}(rz_2)$
with $|z_1|\le2$ and $|z_2|\le2$.

Set
\[
a=\ind_{\{|z_1|\le1\}},
\qquad
b=\ind_{\{|z_2|\le1\}},
\qquad
c=\ind_{\{|z_1-z_2|\le1\}},
\]
and
\[
c^{\mathrm{loc}}_{r,x,\tau}
=
\ind_{\{
d_g(\exp_{x,\tau}(rz_1),\exp_{x,\tau}(rz_2))
\le r
\}},\qquad 
H^{\mathrm{loc}}_{r,x,\tau}
=
\mathfrak h(a,b,c^{\mathrm{loc}}_{r,x,\tau}).
\]
The two anchor-edge indicators are exact, since
\[
d_g\bigl(x,\exp_{x,\tau}(rz_i)\bigr)=r|z_i|.
\]

For a Boolean variable \(c\),
\[
\mathfrak h(a,b,c)
=
\mathfrak h(a,b,0)
+
\bigl[
\mathfrak h(a,b,1)-\mathfrak h(a,b,0)
\bigr]c.
\]
Therefore
\[
\begin{aligned}
H^{\mathrm{loc}}_{r,x,\tau}
-
H_{\mathfrak h}
&=
\widetilde L_{\mathfrak h}(z_1,z_2)
\bigl(
c^{\mathrm{loc}}_{r,x,\tau}-c
\bigr)\\
&=
L_{\mathfrak h}(z_1,z_1-z_2)
\bigl(
c^{\mathrm{loc}}_{r,x,\tau}-c
\bigr).
\end{aligned}
\]

The same graph-distance argument gives
\[
H_{\mathfrak h}(z_1,z_2)\ne0
\quad\text{or}\quad
H^{\mathrm{loc}}_{r,x,\tau}(z_1,z_2)\ne0
\quad\Longrightarrow\quad
|z_1|\le2,\qquad |z_2|\le2.
\]
Moreover,
\[
\ell_{\mathfrak h}(0,0)
=
\mathfrak h(0,0,1)-\mathfrak h(0,0,0)
=
0,
\]
so \(L_{\mathfrak h}\) satisfies the hypothesis of
Lemma~\ref{lem:graph-boundary}.

Choose an \(O(d)\)-invariant cutoff
\[
\chi\in C_c^\infty(\R^d\times\R^d)
\]
that equals one on a neighborhood of
\[
\mathcal K_*
:=
\{(z,w):|z|\le2,\ |z-w|\le2\}.
\]
After decreasing \(r_0\), if necessary, assume
\[
r_0\le r_\chi.
\]
Thus multiplication by \(\chi(z_1,z_1-z_2)\), followed by extension by
zero outside the local normal-coordinate domain, does not change the
anchored integral.

Let \(j_{x,\tau}(v)\) denote the normal-coordinate volume density:
\[
d\mathrm{vol}_g(\exp_{x,\tau}(v))
=
j_{x,\tau}(v)\,dv.
\]
The change of variables
$y=\exp_{x,\tau}(rz_1)$ and $z=\exp_{x,\tau}(rz_2)$
gives
\[
\mathcal I_{\mathfrak h}(r,x)
=
\int_{(\R^d)^2}
\chi(z_1,z_1-z_2)
H^{\mathrm{loc}}_{r,x,\tau}(z_1,z_2)
F_{r,x,\tau}(z_1,z_2)
\,dz_1\,dz_2,
\]
where
\[
F_{r,x,\tau}
:=
f(x)
f(\exp_{x,\tau}(rz_1))
f(\exp_{x,\tau}(rz_2))
j_{x,\tau}(rz_1)
j_{x,\tau}(rz_2).
\]

Since \(\chi=1\) on \(\operatorname{supp}H_{\mathfrak h}\),
\[
\chi H^{\mathrm{loc}}_{r,x,\tau}
=
H_{\mathfrak h}
+
\chi
\bigl(
H^{\mathrm{loc}}_{r,x,\tau}
-
H_{\mathfrak h}
\bigr).
\]
Hence
\[
\begin{aligned}
\mathcal I_{\mathfrak h}(r,x)
={}&
\int_{(\R^d)^2}
H_{\mathfrak h}(z_1,z_2)
F_{r,x,\tau}(z_1,z_2)
\,dz_1\,dz_2\\
&+
\int_{(\R^d)^2}
\chi(z_1,z_1-z_2)
\bigl(
H^{\mathrm{loc}}_{r,x,\tau}(z_1,z_2)
-
H_{\mathfrak h}(z_1,z_2)
\bigr)
F_{r,x,\tau}(z_1,z_2)
\,dz_1\,dz_2.
\end{aligned}
\]
We treat these two terms separately.  The first term accounts for the
Taylor expansion of the density and the normal-coordinate volume density,
while the second is the boundary contribution coming from the
curvature-induced displacement of the internal-edge condition.

In the following expansion, all tensors are evaluated at \(x\) and pulled
back by \(\tau\), and we write \(f=f(x)\).
For \(i=1,2\), consider the radial geodesic
\[
\gamma_i(t):=\exp_{x,\tau}(t z_i).
\]
Since
\[
\gamma_i(0)=x,
\qquad
\dot\gamma_i(0)=\tau z_i,
\qquad
\nabla_{\dot\gamma_i}\dot\gamma_i=0,
\]
Taylor's formula along \(\gamma_i\) gives
\[
f(\exp_{x,\tau}(rz_i))
=
f
+
r\,df(z_i)
+
\frac{r^2}{2}\nabla^2f(z_i,z_i)
+
O(r^3).
\]
Moreover, the normal-coordinate volume density satisfies
\[
j_{x,\tau}(rz_i)
=
1-\frac{r^2}{6}\Ric(z_i,z_i)+O(r^3).
\]
Hence, uniformly on the fixed compact set involved,
\[
\begin{aligned}
F_{r,x,\tau}
={}&
f^3
+
r f^2\{df(z_1)+df(z_2)\}\\
&+
r^2\Bigg[
f\,df(z_1)df(z_2)
+
\frac{f^2}{2}
\{
\nabla^2f(z_1,z_1)
+
\nabla^2f(z_2,z_2)
\}\\
&\hspace{3em}
-
\frac{f^3}{6}
\{
\Ric(z_1,z_1)
+
\Ric(z_2,z_2)
\}
\Bigg]
+
O(r^3).
\end{aligned}
\]

For the remainder of the proof, unqualified integrals are over
\((\R^d)^2\) with respect to \(dz_1\,dz_2\), and we suppress the argument
\(d\) in \(J_{\mathfrak h}(d)\), \(M_{\mathfrak h}(d)\), and
\(\mathcal B_{\mathfrak h}(d)\).

The linear term vanishes under
\[
(z_1,z_2)\longmapsto(-z_1,-z_2).
\]
By Lemma~\ref{lem:three-point-tensor-moments},
\[
\int H_{\mathfrak h}z_1\otimes z_1
=
\frac{M_{\mathfrak h}}dI_d,
\qquad
\int H_{\mathfrak h}z_1\otimes z_2
=
\frac{M_{\mathfrak h}}{2d}I_d.
\]
By symmetry, the same first identity holds with \(z_2\) in place of
\(z_1\).  Therefore
\[
\int
H_{\mathfrak h}
df(z_1)df(z_2)
=
\frac{M_{\mathfrak h}}{2d}
\|\grad f(x)\|_g^2,
\]
\[
\int
H_{\mathfrak h}
\{
\nabla^2f(z_1,z_1)
+
\nabla^2f(z_2,z_2)
\}
=
\frac{2M_{\mathfrak h}}d
\Delta_g f(x),
\]
and
\[
\int
H_{\mathfrak h}
\{
\Ric(z_1,z_1)
+
\Ric(z_2,z_2)
\}
=
\frac{2M_{\mathfrak h}}d
\Scal_g(x).
\]
Together with
\[
\int H_{\mathfrak h}=J_{\mathfrak h},
\]
this gives
\[
\begin{aligned}
\int H_{\mathfrak h}F_{r,x,\tau}
={}&
J_{\mathfrak h}f(x)^3\\
&+
r^2\left\{
\frac{M_{\mathfrak h}}{2d}
f(x)\|\grad f(x)\|_g^2
+
\frac{M_{\mathfrak h}}d
f(x)^2\Delta_g f(x)
-
\frac{M_{\mathfrak h}}{3d}
f(x)^3\Scal_g(x)
\right\}
+
O_H(r^3).
\end{aligned}
\]

We now turn to the second term in the above decomposition.
Using
\[
H^{\mathrm{loc}}_{r,x,\tau}
-
H_{\mathfrak h}
=
L_{\mathfrak h}(z_1,z_1-z_2)
\bigl(
c^{\mathrm{loc}}_{r,x,\tau}-c
\bigr),
\]
it is equal to
\[
\begin{aligned}
&\int_{(\R^d)^2}
\chi(z_1,z_1-z_2)
L_{\mathfrak h}(z_1,z_1-z_2)
\bigl(
c^{\mathrm{loc}}_{r,x,\tau}-c
\bigr)
F_{r,x,\tau}(z_1,z_2)
\,dz_1\,dz_2.
\end{aligned}
\]
On the support of \(\chi\),
\[
F_{r,x,\tau}(z_1,z_2)
=
f(x)^3+O(r)
\]
uniformly in \(x\) and \(\tau\).  On the other hand, the absolute estimate
in Lemma~\ref{lem:graph-boundary} gives
\[
\int_{(\R^d)^2}
\chi(z_1,z_1-z_2)
|L_{\mathfrak h}(z_1,z_1-z_2)|
\bigl|
c^{\mathrm{loc}}_{r,x,\tau}-c
\bigr|
\,dz_1\,dz_2
=
O_H(r^2).
\]
Hence replacing \(F_{r,x,\tau}\) by \(f(x)^3\) in the boundary term
incurs only an \(O_H(r^3)\) error.  Thus
\[
\begin{aligned}
&\int_{(\R^d)^2}
\chi(z_1,z_1-z_2)
\bigl(
H^{\mathrm{loc}}_{r,x,\tau}(z_1,z_2)
-
H_{\mathfrak h}(z_1,z_2)
\bigr)
F_{r,x,\tau}(z_1,z_2)
\,dz_1\,dz_2\\
&\qquad=
f(x)^3
\int_{(\R^d)^2}
\chi(z_1,z_1-z_2)
L_{\mathfrak h}(z_1,z_1-z_2)
\bigl(
c^{\mathrm{loc}}_{r,x,\tau}-c
\bigr)
\,dz_1\,dz_2
+
O_H(r^3).
\end{aligned}
\]

Now make the linear change of variables
\[
z=z_1,
\qquad
w=z_1-z_2.
\]
Its Jacobian has absolute value one.  Lemma~\ref{lem:graph-boundary}
therefore gives
\[
\begin{aligned}
&\int_{(\R^d)^2}
\chi(z_1,z_1-z_2)
\bigl(
H^{\mathrm{loc}}_{r,x,\tau}(z_1,z_2)
-
H_{\mathfrak h}(z_1,z_2)
\bigr)
F_{r,x,\tau}(z_1,z_2)
\,dz_1\,dz_2\\
&\qquad=
\frac{r^2f(x)^3}{6}
\int_{S^{d-1}}\int_{\R^d}
L_{\mathfrak h}(z,w)
R_{x,\tau}(z,w,z,w)
\,dz\,d\sigma(w)
+
o_H(r^2).
\end{aligned}
\]

By Lemma~\ref{lem:rotational-contraction},
\[
\int_{S^{d-1}}\int_{\R^d}
L_{\mathfrak h}(z,w)
R_{x,\tau}(z,w,z,w)
\,dz\,d\sigma(w)
=
\frac{\mathcal B_{\mathfrak h}}{d(d-1)}
\Scal_g(x),
\]
while Proposition~\ref{prop:universal-boundary-moment} gives
\[
\mathcal B_{\mathfrak h}
=
\frac{d-1}{2}M_{\mathfrak h}.
\]
Hence the second term in the decomposition, namely the boundary
contribution, is
\[
\frac{r^2M_{\mathfrak h}}{12d}
f(x)^3\Scal_g(x)
+
o_H(r^2).
\]

Combining the two scalar-curvature contributions,
\[
-\frac{M_{\mathfrak h}}{3d}
+
\frac{M_{\mathfrak h}}{12d}
=
-\frac{M_{\mathfrak h}}{4d},
\]
we obtain
\[
\begin{aligned}
\mathcal I_{\mathfrak h}(r,x)
={}&
J_{\mathfrak h}f(x)^3\\
&+
r^2\Bigg[
\frac{M_{\mathfrak h}}{2d}
f(x)\|\grad f(x)\|_g^2
+
\frac{M_{\mathfrak h}}d
f(x)^2\Delta_g f(x)
-
\frac{M_{\mathfrak h}}{4d}
f(x)^3\Scal_g(x)
\Bigg]\\
&+
\operatorname{Rem}_{\mathfrak h}(r,x),
\end{aligned}
\]
with
\[
\operatorname{Rem}_{\mathfrak h}(r,x)=o_H(r^2)
\]
uniformly in \(x\in M\) and in admissible
\(\mathfrak h\) satisfying \(\|\mathfrak h\|_\infty\le H\).
Indeed, the Taylor remainders above are uniform on the fixed compact
coordinate set, the boundary replacement error is \(O_H(r^3)\), and
Lemma~\ref{lem:graph-boundary} is uniform for bounded
\(\ell_{\mathfrak h}\).  Since
\[
\|\ell_{\mathfrak h}\|_\infty
\le
2\|\mathfrak h\|_\infty,
\]
all of these estimates are uniform over
\(\|\mathfrak h\|_\infty\le H\).

Define
\[
\omega_H(0):=0,
\]
and, for \(0<s\le r_0\),
\[
\omega_H(s)
:=
\sup_{0<r\le s}
\sup_{x\in M}
\sup_{\substack{
\mathfrak h\ \mathrm{admissible}\\
\|\mathfrak h\|_\infty\le H
}}
\frac{
|\operatorname{Rem}_{\mathfrak h}(r,x)|
}{r^2}.
\]
Then \(\omega_H\) is nondecreasing,
\[
\omega_H(s)\longrightarrow0
\qquad\text{as }s\downarrow0,
\]
and
\[
|\operatorname{Rem}_{\mathfrak h}(r,x)|
\le
r^2\omega_H(r).
\]

Finally, exchangeability gives
\[
\E U_{\mathfrak h}(r)
=
\binom{n}{3}r^{2d}
\int_M
\mathcal I_{\mathfrak h}(r,x)
\,d\mathrm{vol}_g(x).
\]
Using
\[
\int_M f^2\Delta_g f\,d\mathrm{vol}_g
=
-2\int_M f\|\grad f\|_g^2\,d\mathrm{vol}_g,
\]
together with
\[
\frac{M_{\mathfrak h}}{2d}
-
\frac{2M_{\mathfrak h}}d
=
-\frac{3M_{\mathfrak h}}{2d},
\qquad
-\frac{M_{\mathfrak h}}{4d}
=
-\frac{3M_{\mathfrak h}}{2d}\cdot\frac16,
\]
yields
\[
\begin{aligned}
\frac{\E U_{\mathfrak h}(r)}
{\binom{n}{3}r^{2d}}
={}&
J_{\mathfrak h}
\int_M f^3\,d\mathrm{vol}_g\\
&-
\frac{3M_{\mathfrak h}}{2d}r^2
\left\{
\int_M f\|\grad f\|_g^2\,d\mathrm{vol}_g
+
\frac16
\int_M f^3\Scal_g\,d\mathrm{vol}_g
\right\}
+
o_H(r^2).
\end{aligned}
\]
Moreover,
\[
\left|
\int_M
\operatorname{Rem}_{\mathfrak h}(r,x)
\,d\mathrm{vol}_g(x)
\right|
\le
\Vol_g(M)r^2\omega_H(r),
\]
so the remainder is independent of \(n\).
\end{proof}

Theorem~\ref{thm:three-point-master} follows directly from
Proposition~\ref{prop:three-point-anchored} by taking
\(H\ge\|\mathfrak h\|_\infty\) and integrating the anchored expansion over
\(x\in M\).  The remainder is independent of \(n\), so the conclusion
continues to hold along arbitrary integer sequences \(n=n(r)\ge3\).

\begin{corollary}
\label{cor:path-triangle}
The Euclidean moments of the induced-path and triangle kernels are given in
Table~\ref{tab:path-triangle-moments}.  Applying
Theorem~\ref{thm:three-point-master} to these two kernels gives the
corresponding expansions for their expected graph counts.  Moreover,
\[
J_{P_3}(d)+J_{K_3}(d)=6A_2(d),
\qquad
M_{P_3}(d)+M_{K_3}(d)=M_2(d).
\]
\end{corollary}

\begin{table}[htbp]
\centering
\caption{Euclidean zeroth and second moments for the induced-path and
triangle kernels.}
\label{tab:path-triangle-moments}
\renewcommand{\arraystretch}{1.35}
\begin{tabular}{@{}c c c@{}}
\toprule
Kernel & $J_G(d)$ & $M_G(d)$ \\
\midrule
$P_3$
& $3\{\omega_d^2-T_0(d)\}$
& $\dfrac{4d}{d+2}\omega_d^2-3T_2(d)$ \\[1.2ex]
$K_3$
& $T_0(d)$
& $T_2(d)$ \\
\bottomrule
\end{tabular}
\end{table}

\begin{proof}
For the triangle kernel,
\[
J_{K_3}(d)
=
\int_{(\R^d)^2}abc\,dz_1\,dz_2
=
T_0(d),
\]
and
\[
M_{K_3}(d)
=
\int_{(\R^d)^2}abc|z_1|^2\,dz_1\,dz_2
=
T_2(d).
\]

For the induced-path kernel,
\[
\mathfrak h_{P_3}
=
ab(1-c)+ac(1-b)+bc(1-a)
=
ab+ac+bc-3abc.
\]
Since each of the three two-edge regions has volume \(\omega_d^2\),
\[
J_{P_3}(d)
=
3\omega_d^2-3T_0(d).
\]

For the second moment,
\[
\int_{(\R^d)^2}ab|z_1|^2\,dz_1\,dz_2
=
\int_{(\R^d)^2}ac|z_1|^2\,dz_1\,dz_2
=
\frac{d}{d+2}\omega_d^2.
\]
For the remaining term, the change of variables
$u=z_2$ and $v=z_1-z_2$
gives
\[
\begin{aligned}
\int_{(\R^d)^2}bc|z_1|^2\,dz_1\,dz_2
&=
\int_{B(0,1)^2}|u+v|^2\,du\,dv\\
&=
\frac{2d}{d+2}\omega_d^2,
\end{aligned}
\]
where the mixed term vanishes by symmetry.  Hence
\[
M_{P_3}(d)
=
\frac{4d}{d+2}\omega_d^2-3T_2(d).
\]

Finally,
\[
\mathfrak h_2
=
\mathfrak h_{P_3}+\mathfrak h_{K_3},
\]
so additivity of the Euclidean moments gives
\[
J_{P_3}(d)+J_{K_3}(d)=J_{\mathfrak h_2}(d)=6A_2(d),
\]
and
\[
M_{P_3}(d)+M_{K_3}(d)=M_{\mathfrak h_2}(d)=M_2(d).
\]
\end{proof}

\begin{lemma}
\label{lem:contrast-nondegeneracy}
For every $d\ge2$,
\[
D_d>0,
\qquad
\gamma_d>0.
\]
\end{lemma}

\begin{proof}
By Table~\ref{tab:path-triangle-moments},
\[
D_d
=\frac{\omega_d^2}{3\{\omega_d^2-T_0(d)\}T_0(d)}
\left\{\frac{4d}{d+2}T_0(d)-3T_2(d)\right\}.
\]
The denominator is positive because both the triangle and induced-path
regions have positive Euclidean measure.  Let $\mu_d$ be the probability
measure on $[0,1]$ given by
\[
\mu_d(ds):=d s^{d-1}\,ds.
\]  Since $s\mapsto s^2$ is
strictly increasing and $s\mapsto L_d(s)$ is strictly decreasing,
Chebyshev's integral inequality for oppositely ordered functions gives
\[
\frac{T_2(d)}{T_0(d)}
=\frac{\int_0^1s^2L_d(s)\,d\mu_d(s)}
       {\int_0^1L_d(s)\,d\mu_d(s)}
<\int_0^1s^2\,d\mu_d(s)
=\frac{d}{d+2}.
\]
Consequently,
\[
\frac{4d}{d+2}T_0(d)-3T_2(d)
>\frac{d}{d+2}T_0(d)>0,
\]
and the assertion follows.
\end{proof}

For $\alpha\in\binom{[n]}{3}$ and $r\ge0$, let
\[
\eta_\alpha(r):=\ind_{\{\Delta(G_n(r)[\alpha])\ge2\}},
\qquad
N_2(r):=\sum_{\alpha\in\binom{[n]}{3}}\eta_\alpha(r),
\]
where $G_n(r)[\alpha]$ is the graph induced by the vertices indexed by
$\alpha$.  Then
\[
\{S_{2,n}>r\}=\{N_2(r)=0\}.
\]

\begin{corollary}
\label{cor:N2-exp}
As $r\downarrow0$,
\[
\begin{aligned}
\E N_2(r)
=\binom{n}{3} r^{2d}\Bigg[&6A_2(d)\int_M f^3\,d\mathrm{vol}_g \\
&-6C_2(d)r^2
\left\{
\int_M f\|\grad f\|_g^2\,d\mathrm{vol}_g
+\frac16\int_M f^3\Scal_g\,d\mathrm{vol}_g
\right\}
+o(r^2)\Bigg].
\end{aligned}
\]
Equivalently,
\[
\begin{aligned}
\E N_2(r)
={}&n^3r^{2d}\Bigg[A_2(d)\int_M f^3\,d\mathrm{vol}_g \\
&\hspace{4.5em}-C_2(d)r^2
\left\{
\int_M f\|\grad f\|_g^2\,d\mathrm{vol}_g
+\frac16\int_M f^3\Scal_g\,d\mathrm{vol}_g
\right\}
+o(r^2)\Bigg]\\
&+O(n^2r^{2d}).
\end{aligned}
\]
\end{corollary}

\begin{proof}
Since
$\mathfrak h_2=\mathfrak h_{P_3}+\mathfrak h_{K_3}$,
Corollary~\ref{cor:path-triangle} gives
\[
J_{\mathfrak h_2}(d)=6A_2(d),
\qquad
M_{\mathfrak h_2}(d)=M_2(d)=4dC_2(d).
\]
Applying Theorem~\ref{thm:three-point-master} to
$\mathfrak h_2(a,b,c)=\ind_{\{a+b+c\ge2\}}$
gives the first formula.  The second follows from
$\binom{n}{3}=\frac{n^3}{6}+O(n^2)$.
\end{proof}

\subsection{Poisson approximation}

Let $\mathcal I_n:=\binom{[n]}{3}$
be the family of all three-element subsets of \([n]\).
For
$\alpha=\{i,j,k\}\in\mathcal I_n$,
recall that
\[
\eta_\alpha(r)
:=
\ind_{\{\Delta(G_n(r)[\alpha])\ge2\}},
\]
where \(G_n(r)[\alpha]\) is the subgraph of \(G_n(r)\) induced by the
vertices indexed by \(\alpha\).
Thus
\[
\eta_\alpha(r)=1
\]
if and only if their induced graph is either a path \(P_3\) or
a triangle \(K_3\).  In particular,
\[
N_2(r)
=
\sum_{\alpha\in\mathcal I_n}\eta_\alpha(r).
\]

To obtain a Poisson approximation for \(N_2(r)\), we need to control the
dependence between two active-triple indicators
\(\eta_\alpha(r)\) and \(\eta_\beta(r)\).
If \(\alpha\cap\beta=\varnothing\), the two indicators depend on disjoint
sets of sample points and hence are independent.  Thus only the cases
\[
|\alpha\cap\beta|=1
\qquad\text{and}\qquad
|\alpha\cap\beta|=2
\]
require estimates.  The following lemma provides the needed overlap
bounds.

\begin{lemma}\label{lem:triple-overlap}
There exists \(C<\infty\) such that, for all sufficiently small \(r\) and
all \(\alpha,\beta\in\mathcal I_n\),
\[
\Pp(\eta_\alpha(r)=1)\le Cr^{2d},
\]
\[
\Pp(\eta_\alpha(r)\eta_\beta(r)=1)\le Cr^{4d}
\qquad (|\alpha\cap\beta|=1),
\]
and
\[
\Pp(\eta_\alpha(r)\eta_\beta(r)=1)\le Cr^{3d}
\qquad (|\alpha\cap\beta|=2).
\]
\end{lemma}

\begin{proof}
Since $M$ is compact and $f$ is bounded, there is a constant $C$ such that
\[
\Pp\{X_j\in B_g(y,2r)\}\le Cr^d
\]
uniformly in $y\in M$ and small $r$.

For one active triple, at least one vertex has degree two.  Taking a union
over the at most three possible choices of this center, fix one such choice
and condition on the location $y$ of the center.  The two remaining points
must both lie in $B_g(y,r)$.  Since they are conditionally independent,
\[
\Pp\{\text{both remaining points lie in }B_g(y,r)\mid y\}
\le C r^{2d}.
\]
After summing over the at most three possible centers and enlarging $C$,
we obtain
\[
\Pp(\eta_\alpha(r)=1)\le Cr^{2d}.
\]

Suppose $|\alpha\cap\beta|=1$, and condition on the common vertex $X_i$.
If both triples are active, each induced three-vertex graph is connected and
has graph diameter at most two.  Hence each of the four non-common vertices is
within geodesic distance at most $2r$ of $X_i$.  Thus all four non-common
vertices lie in $B_g(X_i,2r)$.  Conditioning on $X_i$ and using their
conditional independence gives
\[
\Pp(\eta_\alpha(r)\eta_\beta(r)=1)\le Cr^{4d}.
\]

Suppose $|\alpha\cap\beta|=2$, and condition on one of the shared vertices.
If both triples are active, each induced three-vertex graph is connected and
has graph diameter at most two.  Hence the other shared vertex and the two
non-shared vertices are all within geodesic distance at most $2r$ of the
chosen shared vertex.  Conditioning on that vertex and using the independence
of the other three points gives
\[
\Pp(\eta_\alpha(r)\eta_\beta(r)=1)\le Cr^{3d}.
\]
\end{proof}

\begin{proposition}\label{prop:N2-poisson}
For every fixed $T<\infty$, setting $r_{2,n}(t)=tn^{-3/(2d)}$, one has, as
$n\to\infty$,
\[
d_{\mathrm{TV}}\bigl(\mathcal L(N_2(r_{2,n}(t))),
\operatorname{Po}(\E N_2(r_{2,n}(t)))\bigr)=O(n^{-1/2})
\]
uniformly for $0\le t\le T$.
\end{proposition}

\begin{proof}
Let $\mathcal I_n:=\binom{[n]}{3}$ and use the indicators $\eta_\alpha(r)$
defined in Section~\ref{sec:main-results}.
Use the dependency graph in which $\alpha$ and $\beta$ are adjacent iff
$\alpha\cap\beta\ne\varnothing$ and $\alpha\ne\beta$. 
Since $\eta_\alpha(r)$ depends only on $(X_i)_{i\in\alpha}$, disjoint
triples give independent indicators.  Hence this is a dependency graph.
Lemma~\ref{lem:dep-graph}
gives
\[
\begin{aligned}
d_{\mathrm{TV}}\bigl(\mathcal L(N_2(r)),\operatorname{Po}(\E N_2(r))\bigr)
\le C\Bigg(&\sum_\alpha\sum_{\beta\sim\alpha}\E[\eta_\alpha\eta_\beta]
+\sum_\alpha\sum_{\beta\in N(\alpha)}\E\eta_\alpha\,\E\eta_\beta\Bigg).
\end{aligned}
\]
There are $O(n^5)$ ordered pairs of triples sharing exactly one vertex and
$O(n^4)$ ordered pairs sharing exactly two vertices.  By
Lemma~\ref{lem:triple-overlap},
\[
\sum_\alpha\sum_{\beta\sim\alpha}\E[\eta_\alpha\eta_\beta]
\le C(n^5r^{4d}+n^4r^{3d}).
\]
For the product-of-marginals term, separate the self term from the neighboring term.
Since $|\mathcal I_n|=O(n^3)$ and $\E\eta_\alpha\le Cr^{2d}$,
\[
\sum_\alpha(\E\eta_\alpha)^2\le Cn^3r^{4d}.
\]
The number of ordered neighboring pairs with $\alpha\ne\beta$ is $O(n^5)$, and
therefore
\[
\sum_\alpha\sum_{\substack{\beta\in N(\alpha)\\beta\ne\alpha}}
\E\eta_\alpha\,\E\eta_\beta
\le Cn^5r^{4d}.
\]
Thus the full product-of-marginals term is bounded by
$C(n^3r^{4d}+n^5r^{4d})$, which is $O(n^5r^{4d})$ when
$r=t n^{-3/(2d)}$ with $t\in[0,T]$.
Therefore
\[
d_{\mathrm{TV}}\bigl(\mathcal L(N_2(r)),\operatorname{Po}(\E N_2(r))\bigr)
\le C(n^4r^{3d}+n^5r^{4d}).
\]
For $r=r_{2,n}(t)=tn^{-3/(2d)}$ and $0\le t\le T$, we have
\[
n^4r^{3d}+n^5r^{4d}
=t^{3d}n^{-1/2}+t^{4d}n^{-1}
=O_T(n^{-1/2}).
\]
\end{proof}

\begin{proposition}
\label{prop:N2-zero-count}
Under the standing assumptions, assume \(d>6\), and set
\[
r_{2,n}(t):=t n^{-3/(2d)}.
\]
Then, as \(n\to\infty\), for every \(T<\infty\), uniformly for
\(0\le t\le T\),
\[
\begin{aligned}
\log\Pp\{N_2(r_{2,n}(t))=0\}
={}&
-A_2(d)t^{2d}\int_M f^3\,d\mathrm{vol}_g\\
&+
C_2(d)t^{2d+2}n^{-3/d}
\left\{
\int_M f\|\grad f\|_g^2\,d\mathrm{vol}_g
+
\frac16\int_M f^3\Scal_g\,d\mathrm{vol}_g
\right\}\\
&+
o(n^{-3/d}).
\end{aligned}
\]
\end{proposition}

\begin{proof}
We combine Corollary~\ref{cor:N2-exp},
Proposition~\ref{prop:N2-poisson}, and
Lemma~\ref{lem:zero-transfer}.

For the active-triple weight \(\mathfrak h_2\), let
\(\omega:=\omega_1\) be the nondecreasing modulus supplied by
Proposition~\ref{prop:three-point-anchored}.  Since
\(\|\mathfrak h_2\|_\infty=1\),
\[
\sup_{0\le t\le T}
t^{2d+2}\omega(tn^{-3/(2d)})
\le
T^{2d+2}\omega(Tn^{-3/(2d)})
\longrightarrow0.
\]
Thus the remainder in Corollary~\ref{cor:N2-exp} is uniform for
\(0\le t\le T\).

Substituting
$r=r_{2,n}(t)=t n^{-3/(2d)}$
into Corollary~\ref{cor:N2-exp} gives
\[
\begin{aligned}
\E N_2(r_{2,n}(t))
={}&
A_2(d)t^{2d}\int_M f^3\,d\mathrm{vol}_g\\
&-
C_2(d)t^{2d+2}n^{-3/d}
\left\{
\int_M f\|\grad f\|_g^2\,d\mathrm{vol}_g
+
\frac16\int_M f^3\Scal_g\,d\mathrm{vol}_g
\right\}\\
&+
o(n^{-3/d})
+
O_T(n^{-1}),
\end{aligned}
\]
uniformly for \(0\le t\le T\).  Here the \(O_T(n^{-1})\) term is the
finite-size correction arising from
\[
\binom{n}{3}
=
\frac{n^3}{6}+O(n^2).
\]

Since \(d>6\), we have
$n^{-1}=o(n^{-3/d})$ and $n^{-1/2}=o(n^{-3/d})$.
Thus \(O_T(n^{-1})\) term may be absorbed into the
\(o(n^{-3/d})\) remainder.

For convenience, define
\[
\lambda_0(t)
:=
A_2(d)t^{2d}\int_M f^3\,d\mathrm{vol}_g
\]
and
\[
\lambda_1(t)
:=
-C_2(d)t^{2d+2}
\left\{
\int_M f\|\grad f\|_g^2\,d\mathrm{vol}_g
+
\frac16\int_M f^3\Scal_g\,d\mathrm{vol}_g
\right\}.
\]
Then
\[
\sup_{0\le t\le T}
\left|
\E N_2(r_{2,n}(t))
-
\lambda_0(t)
-
\lambda_1(t)n^{-3/d}
\right|
=
o(n^{-3/d}).
\]

On the other hand, Proposition~\ref{prop:N2-poisson} gives
\[
\sup_{0\le t\le T}
d_{\mathrm{TV}}\!\left(
\mathcal L(N_2(r_{2,n}(t))),
\operatorname{Po}(\E N_2(r_{2,n}(t)))
\right)
=
O(n^{-1/2})
=
o(n^{-3/d}).
\]
The expectation expansion also shows that
\[
\sup_n\sup_{0\le t\le T}
\E N_2(r_{2,n}(t))
<
\infty.
\]
Hence Lemma~\ref{lem:zero-transfer} applies uniformly on \([0,T]\).
For \(0<t\le T\), it follows that
\[
\log\Pp\{N_2(r_{2,n}(t))=0\}
=
-\lambda_0(t)
-\lambda_1(t)n^{-3/d}
+
o(n^{-3/d}).
\]
Substituting the definitions of \(\lambda_0\) and \(\lambda_1\) gives
\[
\begin{aligned}
\log\Pp\{N_2(r_{2,n}(t))=0\}
={}&
-A_2(d)t^{2d}\int_M f^3\,d\mathrm{vol}_g\\
&+
C_2(d)t^{2d+2}n^{-3/d}
\left\{
\int_M f\|\grad f\|_g^2\,d\mathrm{vol}_g
+
\frac16\int_M f^3\Scal_g\,d\mathrm{vol}_g
\right\}\\
&+
o(n^{-3/d}).
\end{aligned}
\]
The remainder is uniform for \(0<t\le T\).
At \(t=0\),
\[
r_{2,n}(0)=0
\]
and \(N_2(0)=0\) almost surely under the continuous sampling law, so
\[
\log\Pp\{N_2(0)=0\}=0.
\]
The displayed expansion is therefore exact at \(t=0\) as well.  This proves
the assertion uniformly on \([0,T]\).
\end{proof}

Theorem~\ref{thm:S2} now follows immediately from
Proposition~\ref{prop:N2-zero-count} and the identity
\[
\{S_{2,n}>r\}=\{N_2(r)=0\}.
\]

\section{Statistical estimation from path--triangle contrasts}
\label{sec:statistical-inference}

This section proves Theorems~\ref{thm:contrast-consistency}
and~\ref{thm:contrast-dense-clt}, and
Corollary~\ref{cor:euler-inference}.  For each $n$, we use the exact Hoeffding decomposition of the order-three
$U$-statistic with radius $r_n$; see, for example,
\cite[Chapter~5]{Serfling1980}.  The point specific to the present problem is
that the normalization $J_{P_3}^{-1}$ versus $J_{K_3}^{-1}$ cancels the
leading one-point projection.

\subsection{Overlap bounds}

For a fixed admissible weight $\mathfrak h$, write
\[
\xi_\alpha=\xi_{\alpha,\mathfrak h}(r),
\qquad
U=U_{\mathfrak h}(r).
\]
Writing $\xi_{\mathfrak h,r}(x,y,z)$ for the corresponding symmetric
three-point kernel, define
\[
\begin{aligned}
m_{\mathfrak h,r}(x)
&:=\E\!\left[
\xi_{\{1,2,3\},\mathfrak h}(r)\mid X_1=x
\right] \\
&=
\int_M\int_M
\xi_{\mathfrak h,r}(x,y,z)
f(y)f(z)\,
d\mathrm{vol}_g(y)\,d\mathrm{vol}_g(z).
\end{aligned}
\]
Let
\[
\theta_{\mathfrak h,r}
:=
\E\xi_{\{1,2,3\},\mathfrak h}(r),
\]
and let $h_{1,r}$, $h_{2,r}$, and $h_{3,r}$ denote the canonical
Hoeffding projections of the symmetric three-point kernel.  Explicitly,
\[
h_{1,r}(x)
=
m_{\mathfrak h,r}(x)-\theta_{\mathfrak h,r},
\]
and, writing
\[
g_{2,r}(x,y)
:=
\E[\xi_{\{1,2,3\},\mathfrak h}(r)
\mid X_1=x,X_2=y],
\]
\[
h_{2,r}(x,y)
=
g_{2,r}(x,y)
-\theta_{\mathfrak h,r}
-h_{1,r}(x)-h_{1,r}(y).
\]
The third projection is defined by
\[
\begin{aligned}
h_{3,r}(x,y,z)
={}&
\xi_{\mathfrak h,r}(x,y,z)-\theta_{\mathfrak h,r}\\
&-h_{1,r}(x)-h_{1,r}(y)-h_{1,r}(z)\\
&-h_{2,r}(x,y)-h_{2,r}(x,z)-h_{2,r}(y,z).
\end{aligned}
\]

We write
\[
R_{\mathfrak h,n}(r)
:=
(n-2)\sum_{1\le i<j\le n}h_{2,r}(X_i,X_j)
+
\sum_{1\le i<j<k\le n}h_{3,r}(X_i,X_j,X_k)
\]
for the remainder after removing the first Hoeffding projection.
The canonical components are mutually orthogonal in $L^2$.

For random variables $Y$ and $Z$, we use the conventions
\[
\operatorname{Var}(Y)
:=
\E\!\left[(Y-\E Y)^2\right]
=
\E[Y^2]-(\E Y)^2
\]
and
\[
\operatorname{Cov}(Y,Z)
:=
\E\!\left[(Y-\E Y)(Z-\E Z)\right]
=
\E[YZ]-\E[Y]\E[Z].
\]
In particular, since
\[
U=\sum_{\alpha\in\binom{[n]}3}\xi_\alpha,
\]
we have
\[
\operatorname{Var}U
=
\sum_{\alpha,\beta\in\binom{[n]}3}
\operatorname{Cov}(\xi_\alpha,\xi_\beta).
\]
The following lemma bounds these covariance contributions according to the
overlap size $|\alpha\cap\beta|$.

\begin{lemma}
\label{lem:statistical-overlap-bound}
For each fixed admissible $\mathfrak h$, there are $C<\infty$ and $r_0>0$
such that, for $n\ge3$ and $0<r<r_0$,
\begin{equation}
\operatorname{Var}U_{\mathfrak h}(r)
\le C\left(n^3r^{2d}+n^4r^{3d}+n^5r^{4d}\right).
\label{eq:general-U-variance}
\end{equation}
If $J_{\mathfrak h}(d)=0$, then the last term improves to
\begin{equation}
\operatorname{Var}U_{\mathfrak h}(r)
\le C\left(n^3r^{2d}+n^4r^{3d}+n^5r^{4d+4}\right).
\label{eq:zero-mass-U-variance}
\end{equation}
Moreover, after removing the first Hoeffding projection, the remainder
$R_{\mathfrak h,n}(r)$ satisfies
\begin{equation}
\operatorname{Var}R_{\mathfrak h,n}(r)
\le C\left(n^3r^{2d}+n^4r^{3d}\right).
\label{eq:Hoeffding-remainder-variance}
\end{equation}
\end{lemma}

\begin{proof}
Since
$U=\sum_{\alpha\in\binom{[n]}3}\xi_\alpha$,
we have
\[
\operatorname{Var}U
=
\sum_{\alpha,\beta\in\binom{[n]}3}
\operatorname{Cov}(\xi_\alpha,\xi_\beta).
\]
We bound the summands according to the overlap size
$|\alpha\cap\beta|$.  Since $\mathfrak h$ is fixed, the corresponding
kernel is uniformly bounded.  We also use throughout the uniform
small-ball estimate
\[
\sup_{x\in M}
\mathbb P\{X_1\in B_g(x,cr)\}
=O(r^d)
\]
for each fixed $c>0$.

If $\alpha\cap\beta=\varnothing$, then $\xi_\alpha$ and $\xi_\beta$
are independent, and hence their covariance is zero.

Suppose first that $\alpha=\beta$.  By connected support, after choosing
one of the three sample points as an anchor, the other two lie in its
geodesic $2r$-ball whenever $\xi_\alpha\ne0$.  Hence
\[
\mathbb P\{\xi_\alpha\ne0\}=O(r^{2d}),
\]
and boundedness of the kernel gives
\[
\operatorname{Var}(\xi_\alpha)
\le \E\xi_\alpha^2
=O(r^{2d}).
\]
There are $\binom n3=O(n^3)$ such pairs.

Suppose next that $|\alpha\cap\beta|=2$.  If both kernel values are
nonzero, the two shared points are at distance at most $2r$, and,
conditional on the shared points, each of the two remaining points lies
in a union of a bounded number of geodesic balls of radius $2r$.
Therefore
\[
\E|\xi_\alpha\xi_\beta|=O(r^{3d}).
\]
Moreover, the preceding one-kernel support bound gives
$\E|\xi_\alpha|=O(r^{2d})$, and hence
\[
\begin{aligned}
|\operatorname{Cov}(\xi_\alpha,\xi_\beta)|
&\le
\E|\xi_\alpha\xi_\beta|
+|\E\xi_\alpha|\,|\E\xi_\beta|\\
&=O(r^{3d}).
\end{aligned}
\]
There are $O(n^4)$ ordered pairs of this type.

If $|\alpha\cap\beta|=1$, relabel the common sample point as $X_1$.
Conditional on $X_1$, the two kernel values are independent, so
\[
\E[\xi_\alpha\xi_\beta\mid X_1]
=
\E[\xi_\alpha\mid X_1]\E[\xi_\beta\mid X_1]
=
m_{\mathfrak h,r}(X_1)^2.
\]
Consequently,
\[
\operatorname{Cov}(\xi_\alpha,\xi_\beta)
=
\operatorname{Var}\bigl(m_{\mathfrak h,r}(X_1)\bigr).
\]
By definition of the normalized anchored integral,
\[
m_{\mathfrak h,r}(x)
=
\frac{r^{2d}}{f(x)}\mathcal I_{\mathfrak h}(r,x).
\]
Since $f$ is bounded below away from zero, the anchored expansion in
Proposition~\ref{prop:three-point-anchored} yields, uniformly in $x$,
\[
m_{\mathfrak h,r}(x)=O(r^{2d}),
\qquad
\operatorname{Var}(m_{\mathfrak h,r}(X_1))=O(r^{4d}).
\]
There are $O(n^5)$ ordered pairs sharing one index.  Combining the four
overlap cases proves \eqref{eq:general-U-variance}.

If $J_{\mathfrak h}(d)=0$, the leading anchored term vanishes and the
same expansion gives
\[
\sup_{x\in M}|m_{\mathfrak h,r}(x)|=O(r^{2d+2}),
\qquad
\operatorname{Var}(m_{\mathfrak h,r}(X_1))=O(r^{4d+4}),
\]
which proves \eqref{eq:zero-mass-U-variance}.

Set
\[
R_{\mathfrak h,n}(r)
=
(n-2)\sum_{1\le i<j\le n}h_{2,r}(X_i,X_j)
+\sum_{1\le i<j<k\le n}h_{3,r}(X_i,X_j,X_k),
\]
and the two canonical sums are orthogonal.

Let
\[
g_{2,r}(x,y)
:=
\E[\xi_{\{1,2,3\},\mathfrak h}(r)
\mid X_1=x,X_2=y].
\]
The projection property gives
\[
\E h_{2,r}(X_1,X_2)^2
\le
\E g_{2,r}(X_1,X_2)^2.
\]
By connected support, $g_{2,r}(x,y)=0$ when $d_g(x,y)>2r$.
When $d_g(x,y)\le2r$, the third point must lie in a union of a bounded
number of geodesic $2r$-balls, and hence
\[
|g_{2,r}(x,y)|=O(r^d)
\]
uniformly.  Therefore
\[
\begin{aligned}
\E g_{2,r}(X_1,X_2)^2
&\le
Cr^{2d}\,
\mathbb P\{d_g(X_1,X_2)\le2r\}\\
&=O(r^{3d}),
\end{aligned}
\]
and thus
\[
\E h_{2,r}^2=O(r^{3d}).
\]
Similarly, by the projection property and the first support estimate,
\[
\E h_{3,r}^2
\le
\E\xi_{\{1,2,3\},\mathfrak h}(r)^2
=
O(r^{2d}).
\]
The exact variance formula for canonical Hoeffding components now gives
\[
\operatorname{Var}R_{\mathfrak h,n}(r)
=
\binom n2(n-2)^2\E h_{2,r}^2
+\binom n3\E h_{3,r}^2
\le
C\{n^4r^{3d}+n^3r^{2d}\},
\]
which is \eqref{eq:Hoeffding-remainder-variance}.
\end{proof}

Define the path--triangle contrast weight by
\[
\mathfrak c_d
:=
\frac{\mathfrak h_{P_3}}{J_{P_3}(d)}
-
\frac{\mathfrak h_{K_3}}{J_{K_3}(d)},
\]
and set
\[
C_n(r)
:=
U_{\mathfrak c_d}(r)
=
\frac{U_{P_3}(r)}{J_{P_3}(d)}
-
\frac{U_{K_3}(r)}{J_{K_3}(d)}.
\]
Since $J_{\mathfrak h}(d)$ is linear in $\mathfrak h$,
\[
J_{\mathfrak c_d}(d)
=
\frac{J_{P_3}(d)}{J_{P_3}(d)}
-
\frac{J_{K_3}(d)}{J_{K_3}(d)}
=
0.
\]
We normalize the contrast by
\[
\widehat{\mathcal G}_{3,n}(r)
:=
-\frac{C_n(r)}
{\gamma_d\binom n3 r^{2d+2}}.
\]

For clarity, we restate Theorem~\ref{thm:contrast-consistency} in the
following equivalent form.

\begin{theorem}
\label{thm:contrast-consistency-revise}
There exist constants $C<\infty$ and $r_0>0$ such that, for every
$n\ge3$ and $0<r<r_0$, the path--triangle contrast satisfies
\begin{equation}
\operatorname{Var} C_n(r)
\le
C\left(
n^3r^{2d}
+n^4r^{3d}
+n^5r^{4d+4}
\right).
\end{equation}
Accordingly, its normalized estimator obeys
\begin{equation}
\operatorname{Var}\widehat{\mathcal G}_{3,n}(r)
\le
C\left(
\frac1{n^3r^{2d+4}}
+
\frac1{n^2r^{d+4}}
+
\frac1n
\right).
\end{equation}

Now let $r_n\downarrow0$.  If the bandwidth sequence satisfies
\begin{equation}
n^3r_n^{2d+4}\longrightarrow\infty,
\qquad
n^2r_n^{d+4}\longrightarrow\infty,
\end{equation}
then the normalized contrast consistently estimates the intrinsic functional:
\[
\widehat{\mathcal G}_{3,n}(r_n)
\xrightarrow{\Pp}
\mathcal G_3(f;g).
\]
\end{theorem}

\begin{proof}
Apply Lemma~\ref{lem:statistical-overlap-bound} to $\mathfrak c_d$; its
Euclidean mass is zero.  This gives
\eqref{eq:contrast-variance-bound}.

By the definition
\[
\widehat{\mathcal G}_{3,n}(r)
=
-\frac{C_n(r)}
{\gamma_d\binom n3 r^{2d+2}},
\]
we have
\[
\operatorname{Var}\widehat{\mathcal G}_{3,n}(r)
=
\frac{\operatorname{Var}C_n(r)}
{\gamma_d^2\binom n3^2r^{4d+4}}.
\]
Since $\binom n3\asymp n^3$ for $n\ge3$,
\eqref{eq:contrast-variance-bound} yields
\[
\operatorname{Var}\widehat{\mathcal G}_{3,n}(r)
\le
C\left(
\frac1{n^3r^{2d+4}}
+
\frac1{n^2r^{d+4}}
+
\frac1n
\right),
\]
which is \eqref{eq:G3-estimator-variance-bound}.

By \eqref{eq:contrast-mean},
\[
\E\widehat{\mathcal G}_{3,n}(r_n)
\longrightarrow\mathcal G_3(f;g).
\]
Furthermore, the right-hand side of
\eqref{eq:G3-estimator-variance-bound} tends to zero under
\eqref{eq:contrast-consistency-bandwidth}.  Chebyshev's inequality completes
the proof.
\end{proof}

\subsection{The H\'ajek-projection-dominant limit}

Let
\[
\theta_r:=\E\xi_{\{1,2,3\},\mathfrak c_d}(r),
\qquad
h_{1,r}(x):=m_{\mathfrak c_d,r}(x)-\theta_r.
\]
The Hoeffding decomposition for the unnormalized order-three statistic is
\begin{equation}
C_n(r)-\E C_n(r)
=\binom{n-1}{2}\sum_{i=1}^nh_{1,r}(X_i)+R_n(r),
\label{eq:contrast-Hoeffding-decomposition}
\end{equation}
where $R_n(r)$ is orthogonal to the first projection.

Since
\[
m_{\mathfrak c_d,r}(x)
=\frac{r^{2d}}{f(x)}\mathcal I_{\mathfrak c_d}(r,x)
\]
and $J_{\mathfrak c_d}(d)=0$, the pointwise expansion in
Proposition~\ref{prop:three-point-anchored} gives, uniformly in $x\in M$,
\begin{equation}
\frac{m_{\mathfrak c_d,r}(x)}{D_dr^{2d+2}}
=q_f(x)+o(1).
\label{eq:conditional-contrast-expansion}
\end{equation}
Integrating against $f\,d\mathrm{vol}_g$ and using
\eqref{eq:contrast-mean} gives
\begin{equation}
\frac{h_{1,r}(x)}{D_dr^{2d+2}}
=\psi_f(x)+o(1)
\label{eq:first-projection-limit}
\end{equation}
uniformly in $x$.

For clarity, we record the following equivalent reformulation of
Theorem~\ref{thm:contrast-dense-clt}, emphasizing the root-$n$ fluctuation
of the normalized contrast around its exact expectation.

\begin{theorem}
\label{thm:contrast-dense-clt-revise}
Let $r_n\downarrow0$ and assume that
\[
n r_n^{d+4}\longrightarrow\infty.
\]
Then the exact-expectation-centered normalized path--triangle contrast has
the root-$n$ limit
\[
\sqrt n\left(
\widehat{\mathcal G}_{3,n}(r_n)
-\E\widehat{\mathcal G}_{3,n}(r_n)
\right)
\xrightarrow{\mathcal D}
N(0,\sigma_f^2),
\]
where
\[
\sigma_f^2
=
4d^2\int_M\psi_f(x)^2f(x)\,d\mathrm{vol}_g(x).
\]
The limiting normal distribution is understood to be degenerate at zero
when $\sigma_f^2=0$.
\end{theorem}

\begin{proof}
Since
\[
\frac{\binom{n-1}{2}}{\binom n3}=\frac3n,
\qquad
\frac{3D_d}{\gamma_d}=2d,
\]
the first term in \eqref{eq:contrast-Hoeffding-decomposition} contributes
\[
-\frac{2d}{n}\sum_{i=1}^n\psi_f(X_i)+o_{L^2}(n^{-1/2})
\]
to
$\widehat{\mathcal G}_{3,n}(r)-
\E\widehat{\mathcal G}_{3,n}(r)$.  Indeed, the uniform remainder in
\eqref{eq:first-projection-limit} has centered $L^2$ norm tending to zero.
Because $\psi_f$ is bounded on the compact manifold and the centered
remainder is uniformly $o(1)$ in $L^2(f\,d\mathrm{vol}_g)$, the triangular
array differs in $L^2$ from the fixed i.i.d.\ sum
$-2d n^{-1}\sum_{i=1}^n\psi_f(X_i)$.  The ordinary i.i.d.\ central limit
theorem therefore gives the normal limit with variance
\eqref{eq:G3-asymptotic-variance}.

It remains to check that the higher Hoeffding terms are negligible.  By
\eqref{eq:Hoeffding-remainder-variance},
\[
\begin{aligned}
&n\,\operatorname{Var}\left(
\frac{R_n(r_n)}
{\gamma_d\binom n3r_n^{2d+2}}
\right)\\
&\hspace{3em}\le
C\left(
\frac1{n^2r_n^{2d+4}}+
\frac1{nr_n^{d+4}}
\right)\longrightarrow0.
\end{aligned}
\]
For the first term, note that
\[
n^2r_n^{2d+4}
=
(n r_n^d)(n r_n^{d+4})
\longrightarrow\infty,
\]
because
\[
n r_n^d
=
(n r_n^{d+4})r_n^{-4}
\longrightarrow\infty.
\]
This proves \eqref{eq:G3-dense-clt}.
\end{proof}

\begin{remark}[Degeneracy under uniform sampling]
Suppose $f\equiv V^{-1}$, where $V=\Vol_g(M)$, and put
\[
\overline{\Scal}_g:=\frac1V\int_M\Scal_g\,d\mathrm{vol}_g.
\]
Then
\[
\psi_f(x)=\frac{\overline{\Scal}_g-\Scal_g(x)}{4dV^2}
\]
and hence
\begin{equation}
\sigma_f^2
=\frac1{4V^5}\int_M
\bigl(\Scal_g-\overline{\Scal}_g\bigr)^2\,d\mathrm{vol}_g.
\label{eq:uniform-asymptotic-variance}
\end{equation}
Thus the root-$n$ limit is nondegenerate exactly when the scalar curvature is
not constant.  In particular, for constant-curvature surfaces the theorem
still holds, but its limit at the root-$n$ scale is degenerate; the consistency
statement in Corollary~\ref{cor:euler-inference} is unaffected.  This
also shows that the root-$n$ fluctuation around the exact mean records spatial
variation of scalar curvature, rather than the Euler characteristic itself.
\end{remark}

\subsection{Surface Euler characteristic}

Throughout this subsection, assume $d=2$ and uniform sampling,
\[
f\equiv V^{-1},
\qquad
V:=\Vol_g(M).
\]
By Table~\ref{tab:path-triangle-moments} and the elementary lens integrals,
\[
J_{P_3}(2)=M_{P_3}(2)=\frac{9\sqrt3\pi}{4},
\qquad
J_{K_3}(2)=\frac{\pi(4\pi-3\sqrt3)}4,
\]
and
\[
\gamma_2=\frac{\pi}{4\pi-3\sqrt3}.
\]

Define the triangle-based estimator
\[
\widehat F_{3,n}(r)
:=
\frac{U_{K_3}(r)}
{\binom n3 r^4J_{K_3}(2)}
\]
and the corresponding volume estimator
\[
\widehat V_n(r):=
\begin{cases}
\widehat F_{3,n}(r)^{-1/2},
&\widehat F_{3,n}(r)>0,\\
1,
&\widehat F_{3,n}(r)=0.
\end{cases}
\]
Define also
\[
\widehat\chi_n(r)
:=
\frac{3}{2\pi}\widehat V_n(r)^3
\widehat{\mathcal G}_{3,n}(r)
=
-\frac{3(4\pi-3\sqrt3)}{2\pi^2}
\frac{\widehat V_n(r)^3C_n(r)}
{\binom n3r^6}.
\]

Finally, let
\[
\mathcal E_{\mathrm{all}}
:=
\{2,1,0,-1,-2,\ldots\},
\]
and define
\[
\widehat\chi_n^{\mathrm{all}}(r)
\in
\operatorname*{argmin}_{k\in\mathcal E_{\mathrm{all}}}
|\widehat\chi_n(r)-k|,
\]
with deterministic tie-breaking.

For clarity, we record the following equivalent reformulation of
Corollary~\ref{cor:euler-inference}.

\begin{corollary}
\label{cor:euler-inference-revise}
Let $r_n\downarrow0$ and assume that
\[
n^2r_n^6\longrightarrow\infty.
\]
Then
\[
\widehat V_n(r_n)\xrightarrow{\Pp}V,
\qquad
\widehat\chi_n(r_n)\xrightarrow{\Pp}\chi(M).
\]
Moreover,
\[
\Pp\{
\widehat\chi_n^{\mathrm{all}}(r_n)=\chi(M)
\}
\longrightarrow1.
\]
Thus projection onto the admissible Euler characteristics gives exact
recovery with probability tending to one.
\end{corollary}

\begin{proof}
The values of $J_{P_3}(2)$ and $J_{K_3}(2)$ follow from
Table~\ref{tab:path-triangle-moments} and the elementary lens integrals
\[
T_0(2)=\frac{\pi(4\pi-3\sqrt3)}4,
\qquad
T_2(2)=\frac{\pi(8\pi-9\sqrt3)}{12}.
\]
Indeed, since $\omega_2=\pi$,
\[
\begin{aligned}
J_{P_3}(2)
&=
3\bigl(\pi^2-T_0(2)\bigr)
=
\frac{9\sqrt3\pi}{4},\\
M_{P_3}(2)
&=
2\pi^2-3T_2(2)
=
\frac{9\sqrt3\pi}{4},
\end{aligned}
\]
while
\[
J_{K_3}(2)=T_0(2)
=
\frac{\pi(4\pi-3\sqrt3)}4,
\qquad
M_{K_3}(2)=T_2(2)
=
\frac{\pi(8\pi-9\sqrt3)}{12}.
\]
Hence
\[
\begin{aligned}
D_2
&=
\frac{M_{P_3}(2)}{J_{P_3}(2)}
-
\frac{M_{K_3}(2)}{J_{K_3}(2)}\\
&=
1-
\frac{8\pi-9\sqrt3}{3(4\pi-3\sqrt3)}
=
\frac{4\pi}{3(4\pi-3\sqrt3)}.
\end{aligned}
\]
Therefore
\[
\gamma_2
=
\frac{3D_2}{4}
=
\frac{\pi}{4\pi-3\sqrt3}.
\]

Since $f\equiv V^{-1}$, we have
\[
\int_M f^3\,d\mathrm{vol}_g=V^{-2},
\qquad
\grad f=0.
\]
Hence the triangle expectation expansion gives
\[
\E\widehat F_{3,n}(r)=V^{-2}+O(r^2).
\]
Since $d=2$, Lemma~\ref{lem:statistical-overlap-bound},
applied to the triangle kernel, gives
\[
\operatorname{Var}U_{K_3}(r)
\le
C\left(
n^3r^4+n^4r^6+n^5r^8
\right).
\]
Therefore,
\[
\begin{aligned}
\operatorname{Var}\widehat F_{3,n}(r)
&\le
C\frac{
n^3r^4+n^4r^6+n^5r^8
}{
\binom n3^2r^8
}\\
&\le
C\left(
\frac1{n^3r^4}
+\frac1{n^2r^2}
+\frac1n
\right),
\end{aligned}
\]
where we used $\binom n3\asymp n^3$ and absorbed the fixed positive
constant $J_{K_3}(2)^{-2}$ into $C$.
Condition~\eqref{eq:surface-consistency-bandwidth} implies
$n r_n^2\to\infty$, and therefore
\[
n^3r_n^4\to\infty,
\qquad
n^2r_n^2\to\infty,
\qquad
n^3r_n^8
=
(n^2r_n^6)(n r_n^2)
\to\infty.
\]
Since $\E\widehat F_{3,n}(r_n)\to V^{-2}$ and $\operatorname{Var}\widehat F_{3,n}(r_n)\to0$,
we have
\[
\widehat F_{3,n}(r_n)\xrightarrow{\Pp}V^{-2}.
\]
The continuous mapping theorem then yields
\[
\widehat V_n(r_n)\xrightarrow{\Pp}V.
\]
The two conditions of Theorem~\ref{thm:contrast-consistency} also hold for
$d=2$, so
\[
\widehat{\mathcal G}_{3,n}(r_n)
\xrightarrow{\Pp}
\mathcal G_3(f;g).
\]
Since $\mathcal G_3(f;g)
=
\frac{2\pi\chi(M)}{3V^3}$,
Slutsky's theorem gives
\[
\widehat\chi_n(r_n)
\xrightarrow{\Pp}
\chi(M).
\]
The Euler characteristics of closed connected surfaces belong to
$\mathcal E_{\mathrm{all}}$, and distinct values in
$\mathcal E_{\mathrm{all}}$ are separated by distance one.  Therefore,
consistency of $\widehat\chi_n(r_n)$ implies
\[
\Pp\{
\widehat\chi_n^{\mathrm{all}}=\chi(M)
\}
\longrightarrow1.
\]
\end{proof}

\section{Discussion and limitations}
\label{sec:discussion}

The path--triangle contrast identifies the intrinsic functional
\[
\mathcal G_3(f;g)
=
\int_M f\|\grad f\|_g^2\,d\mathrm{vol}_g
+\frac16\int_M f^3\Scal_g\,d\mathrm{vol}_g.
\]
Under uniform sampling on a closed surface, Gauss--Bonnet reduces the
second-order geometric term to a multiple of $\chi(M)$, leading to the
estimator in Corollary~\ref{cor:euler-inference}.  In general, however,
path and triangle counts alone do not separate the density-gradient and
scalar-curvature contributions.

The expectation expansion gives an $o(1)$ normalized bias, which is
sufficient for consistency, whereas the central limit theorem is centered
at the exact expectation.  A target-centered limit theorem, studentization,
or confidence intervals would require sharper bias and variance estimates.

The restriction $d>6$ in the degree-two threshold theorem comes from the
separation between the overlap error $O(n^{-1/2})$ and the geometric
correction $n^{-3/d}$.  At $d=6$, determining the next term would require a
more refined connected-cluster or factorial-cumulant analysis.

Finally, the boundaryless assumption avoids order-$r$ boundary corrections.
The argument also exploits the special structure of connected
three-vertex kernels; for $k\ge3$, several moving internal-chord boundaries
appear, and the present two-parameter reduction need not persist.

\appendix

\section{One-dimensional reduction of the \texorpdfstring{$k=2$}{k=2} constants}
\label{app:k2-one-dimensional}
Assume \(d\ge2\) throughout this appendix.
This appendix gives the elementary one-dimensional calculation underlying
the positive representation of \(C_2(d)\).
Equivalently, since \(M_2(d)=4d\,C_2(d)\), put
\[
J_d=\int_0^{1/2}(1-x^2)^{(d+1)/2}\,dx.
\]
We prove the manifestly positive representation
\[
M_2(d)=\frac{8d\omega_d\omega_{d-1}}{d+1}J_d.
\]
Recall that
\[
M_2(d)
=
\frac{4d}{d+2}\omega_d^2
-
2d\omega_d
\int_0^1 s^{d+1}L_d(s)\,ds.
\]
Also recall that
\[
L_d(s)=2\omega_{d-1}\int_{s/2}^{1}h(u)\,du,
\qquad h(u)=(1-u^2)^{(d-1)/2}.
\]
Fubini--Tonelli theorem implies that
\[
\begin{aligned}
\int_0^1s^{d+1}L_d(s)\,ds
&=2\omega_{d-1}\int_0^1s^{d+1}\int_{s/2}^{1}h(u)\,du\,ds \\
&=\frac{2\omega_{d-1}}{d+2}
\left[\int_0^{1/2}(2u)^{d+2}h(u)\,du+
\int_{1/2}^{1}h(u)\,du\right].
\end{aligned}
\]
Let
\[
I_0=\int_0^{1/2}h(u)\,du,
\qquad I_1=\int_{1/2}^{1}h(u)\,du,
\qquad I_2=\int_0^{1/2}u^{d+2}h(u)\,du.
\]
Since $\omega_d=2\omega_{d-1}(I_0+I_1)$, the expression for $M_2(d)$ becomes
\[
M_2(d)=\frac{4d\omega_d\omega_{d-1}}{d+2}
\left(2I_0+I_1-2^{d+2}I_2\right).
\]
It remains to identify the bracket.  The substitutions $u=\cos\theta$ and
$y=\sin(\theta/2)$ give
\[
I_1=\int_0^{\pi/3}\sin^d\theta\,d\theta
=2^{d+1}\int_0^{1/2}y^d h(y)\,dy.
\]
Let
\[
P_d=\int_0^{1/2}u^dh(u)\,du,
\qquad
F_d=\frac12\left(\frac34\right)^{(d+1)/2}.
\]
Integrating
\[
\frac{d}{du}\left(u^{d+1}(1-u^2)^{(d+1)/2}\right)
=(d+1)u^d(1-u^2)^{(d+1)/2}-(d+1)u^{d+2}h(u)
\]
over $[0,1/2]$ gives
\[
P_d-2I_2=\frac{2^{-d}F_d}{d+1}.
\]
Thus
\[
I_1-2^{d+2}I_2=2^{d+1}(P_d-2I_2)=\frac{2F_d}{d+1}.
\]
On the other hand, integrating
\[
\frac{d}{du}\left(u(1-u^2)^{(d+1)/2}\right)
=(1-u^2)^{(d+1)/2}-(d+1)u^2h(u)
\]
over $[0,1/2]$ yields
\[
F_d=(d+2)J_d-(d+1)I_0.
\]
Consequently, we have
\[
2I_0+I_1-2^{d+2}I_2
=2I_0+\frac{2F_d}{d+1}
=\frac{2(d+2)}{d+1}J_d,
\]
and hence
\[
M_2(d)=\frac{8d\omega_d\omega_{d-1}}{d+1}J_d
\]
holds.

\begingroup
\small
\emergencystretch=3em
\sloppy
\bibliographystyle{abbrvurl}
\bibliography{references_cited_only}
\endgroup

\section*{Statements and Declarations}

\noindent\textbf{Funding.}
This work was supported by the Japan Society for the Promotion of Science
(JSPS) KAKENHI Grant Number 23K12507.

\medskip
\noindent\textbf{Competing interests.}
The authors declare that they have no competing interests.

\medskip
\noindent\textbf{Author contributions.}
All authors contributed to the conception of the study, the mathematical
analysis, and the preparation of the manuscript. All authors read and approved
the final manuscript.

\medskip
\noindent\textbf{Data availability.}
No datasets were generated or analysed during the current study.

\end{document}